\documentclass[12pt]{amsart}

\usepackage[utf8]{inputenc}
\usepackage[english]{babel}
\usepackage{hyperref}
\usepackage{comment}

\usepackage{amssymb,amsthm,amsmath,mathrsfs,amssymb}

\usepackage[mathscr]{euscript}

\usepackage{enumitem}
\makeatletter
\def\paragraph{\@startsection{paragraph}{4}%
	\z@\z@{-\fontdimen2\font}%
	{\normalfont\bfseries}}
\makeatother

\usepackage{xcolor}

\newcommand{\eps}{\varepsilon}

\newcommand{\C}{\mathbb{C}}
\newcommand{\Disk}{\mathbb{D}}

\newcommand{\N}{\mathbb{N}}
\newcommand{\Q}{\mathbb{Q}}

\newcommand{\R}{\mathbb{R}}

\newcommand*\colvec[3][]{
	\begin{pmatrix}\ifx\relax#1\relax\else#1\\\fi#2\\#3\end{pmatrix}
}
\newcommand*\abs[1]{
	\left|#1\right|
}

\newcommand*\deriv{
	\mathrm{d}
}

\newcommand*\Leb{\
	\mathrm{L}
}

\newcommand*\dist{
	\mathrm{dist}
}
\newcommand*\diam{
	\mathrm{diam}
}
\newcommand\restr[2]{{
		\left.\kern-\nulldelimiterspace 
		#1 
		\littletaller 
		\right|_{#2} 
	}}
	\newcommand{\littletaller}{\mathchoice{\vphantom{\big|}}{}{}{}}
	
	\newcommand{\norm}[1]{\left\lVert#1\right\rVert}
	\newtheorem{thm}{Theorem}
	\newtheorem*{thm*}{Theorem}
	\newtheorem{lem}{Lemma}[section]

	\newtheorem{coro}{Corollary}[section]
	\newtheorem{prop}{Proposition}
	
	\theoremstyle{definition}
	
	\newtheorem*{rem}{Remark}
	\newtheorem*{assumption}{Assumption}
	\newtheorem{defn}{Definition}
	\newtheorem{notation}{Notation}
	\newtheorem{claim}{Claim}
	\renewcommand{\leq}{\leqslant}
	\renewcommand{\geq}{\geqslant}
	\normalsize \linespacing=\baselineskip
	
	\evensidemargin=.65in{}
	\usepackage{fullpage}

\newcommand\newlink[2]{{\protect\hyperlink{#1}{#2}}}
    
	\newcommand*\epsLemProjLocale{
		\newlink{lem:proj-locale}{\eps_\text{proj}}
        	}
	\newcommand*\rhoSection{
		\newlink{def:rho-section}{\rho_0}
        }
	\newcommand*\rLemProjLocale{
		r_0
	}
	\newcommand*\RLemProjLocale{
		R_0
	}
	\newcommand*\lambdaLemProjLocale{
		\lambda
	}
	
	\newcommand*\CstLemProjGlobale{
		\newlink{lem:proj-globale}{C_0}
	}

	\newcommand*\RLemSautBrownien{
		\newlink{lem:saut-brownien}{R_\text{min}}
	}
	
	\newcommand*\kappaLemProjLocale{
		\kappa
	}
	\newcommand*\CstALemComparaison{
		C_3
	}
	\newcommand*\CstBLemComparaison{
		C_4
	}
	\newcommand*\epsLemComparaison{
		\eps_1
	}
	\newcommand*\CstDefnA{
		C_5
	}
	\newcommand*\CstLemProbaA{
		C_6
	}
	\newcommand*\CstLemAbsoluteContinuite{
		C_7
	}
	\newcommand*\CstLemRedDelta{
		C_8
	}
	\newcommand*\CstCoroGenericite{
		C_9}
	\newcommand*\cstLemSautBrownien{
		c_0
	}
	\newcommand*\cstLemComparaison{
		c_1
	}
	\newcommand*\cstLemProbaA{
		c_2
	}
	\newcommand*\cstLemAbsoluteContinuite{
		c_3
	}
    \makeatletter
\newcommand\newtarget[2]{\Hy@raisedlink{\hypertarget{#1}{}}#2}
\makeatother

	\title{Unique ergodicity for singular holomorphic foliations of $\mathbb{P}^3(\C)$ with an invariant plane}
	\author{Félix Lequen}
    \address{Department of Mathematics, Statistics, and Computer Science, University of Illinois Chicago}
\email{lequen@uic.edu}
	
\begin{document}
		\maketitle
		\begin{abstract}
			We prove a unique ergodicity theorem for singular holomorphic foliations of $\mathbb{P}^3(\C)$ with hyperbolic singularities and with an invariant plane with no foliation cycle, in analogy with a result of Dinh-Sibony concerning unique ergodicity for foliations of $\mathbb{P}^2(\C)$ with an invariant line. The proof is dynamical in nature and adapts the work of Deroin-Kleptsyn to a singular context, using the fundamental integrability estimate of Nguyên.
		\end{abstract}

\section{Introduction}
\subsection{Context and motivation. }Consider two polynomials $P$, $Q$ on $\C^2$ and the associated algebraic ordinary differential equation in the complex domain
\begin{equation}
\frac{\deriv x}{\deriv t} = P(x, y),\;\frac{\deriv y}{\deriv t} = Q(x, y).
				\label{eq2}
				\end{equation}
				The study of the dynamics of the solutions of this equation is a largely open field. In the case of generic $P, Q$, Khuda{\u{\i}}-Verenov \cite{KhudaiVerenov} has shown that the trajectories are dense in $\C^2$. This was extended by Ilyashenko \cite{Ilyashenko} who proved that moreover, generically, these foliations are ergodic and rigid. In higher dimensions, analogous, though weaker, results have been obtained by Loray-Rebelo \cite{LorayRebelo}. 
				
				Another interesting question is the study of the statistical behaviour of solutions, from the point of view of ergodic theory. In this direction, Dinh-Sibony \cite{DinhSibony} have shown that non-stationary solutions of the equation \ref{eq2} spend, in a certain sense, most of their time near infinity — a fact that is perhaps surprising given that the leaves are dense. This is expressed in the following way: for generic equations, solutions of \ref{eq2} are Riemann surfaces which are hyperbolic, that is, they are uniformised by the Poincaré disk $\mathbb{D}$. In particular, using the hyperbolic metric, we can consider a sequence of measures $m_R$, obtained by pushing forward by a uniformisation map the volume measure on the hyperbolic ball $B(0, R)$, for $R > 0$, and multiplying by a certain normalising factor depending on $R$. This is similar to the usual Krylov-Bogoliubov construction in ergodic theory. Then their result implies that the sequence $m_R$ (more precisely, a variant of it with a logarithmic weight) escapes at infinity as $R \to \infty$.
				
				Considering the equation \ref{eq2} in the projective plane $\mathbb{P}^2(\C)$, we can reformulate (generically) the problem as the study of holomorphic vector fields in $\mathbb{P}^2(\C)$ with an invariant line as a unique invariant algebraic curve. More precisely, we will be interesed in the geometric structure associated to the equation, that is, in the holomorphic foliation it defines, given by the partition of the ambient space in trajectories (leaves) and singular points. 
				
				A useful notion to prove the theorem mentioned above, and more generally to study ergodic properties of holomorphic foliations, is that of \emph{harmonic measure} for a foliation, which can very loosely be seen as a weak analogue, for a holomorphic foliation, of an invariant measure for a usual dynamical system (we will see below a way to make this analogy more precise). These are the measures that are invariant by a diffusion process on the leaves of the foliation. Then one is interested in the classification of such measures. Loosely speaking, the theorem of Dinh-Sibony asserts that there is a unique such harmonic measure in the context of equation \ref{eq2}. The focus of this work will be to prove a higher-dimensional analogue of the result of Dinh-Sibony.
				\subsection{Statement and related results} We now give more precise statements, using terminology to be defined precisely in the next section. The theorem of Dinh-Sibony mentioned above can be formulated as follows:
				
				\begin{thm}[Dinh-Sibony \cite{DinhSibony}]
					Let $\mathscr{F}$ be a singular holomorphic foliation by curves on $\mathbb{P}^2(\C)$, seen as a compactification of $\C^2$. Assume that all the singularities are hyperbolic and that the only invariant algebraic curve is the line at infinity $L_\infty$. Then there is a unique positive harmonic current directed by $\mathscr{F}$, up to multiplication by a constant, given by the current of integration on $L_\infty$.
				\end{thm}
				In this article, we study holomorphic foliations in $\mathbb{P}^3(\C)$ which have an invariant projective plane $P$. To describe the result, let us first quote a similar result of Fornæss-Sibony:
				
\begin{thm}[Fornæss-Sibony]
Let $\mathscr{F}$ be a singular holomorphic foliation by curves on $\mathbb{P}^2(\C)$. Assume that all the singularities are hyperbolic and that there is no invariant algebraic curve. Then there is a unique positive harmonic current directed by $\mathscr{F}$, up to multiplication by a constant.
\label{thm:fornaess-sibony}
\end{thm}
Let us also mention its generalisation to compact Kähler surfaces:
\begin{thm}[Dinh-Nguyên-Sibony \cite{DNSKahler}]
Consider a compact Kähler surface $S$ and a singular holomorphic foliation $\mathscr{F}$ with only hyperbolic singularities, and no foliation cycle directed by it. Then there exists a unique harmonic current directed by $\mathscr{F}$, up to multiplication by a constant.
				\end{thm}
				In this article, we prove the following theorem:
				\begin{thm}
					Let $\mathscr{F}$ be a singular holomorphic foliation by curves on $\mathbb{P}^3(\C)$, seen as a compactification of $\C^3$. Assume that there is only a finite number of singularities and that they are all hyperbolic, that the plane at infinity $P$ is invariant by $\mathscr{F}$ and that there are no algebraic curves invariant by $\mathscr{F}$. Then there is a unique positive harmonic current directed by $\mathscr{F}$, up to multiplication by a constant, and it is the current given by the theorem of Fornæss-Sibony \ref{thm:fornaess-sibony} applied to $P$ and $\restr{\mathscr{F}}{P}$, extended by zero.
					\label{lethm}
				\end{thm}
				Note that the absence of foliation cycles directed by $\mathscr{F}$ implies that there are no invariant algebraic curves, by considering the current of integration on them. In dimension $2$, with hyperbolic singularities, the converse is also true, by the work of Camacho-Sad-Lins Neto \cite{CSLN}. For $\mathbb{P}^3(\C)$ and hyperbolic singularities, it is a consequence of the work of Brunella \cite{Brunella}. The condition on the foliation is therefore generic, see Lins Neto-Soares \cite[section 4]{LinsNetoSoares} and Coutinho \cite[theorem 4.2]{Coutinho} for the condition on invariant algebraic curves.

                \begin{rem}
                    It would possible to prove a similar statement for more general complex projective manifolds of dimension $3$ with a hypersurface $P$ such that $[P]$ is an ample divisor. However for simplicity we stick to the case of $\mathbb{P}^3(\C)$ with a hyperplane.
                \end{rem}

				Let $p \in \mathbb{P}^3(\C)$ be a regular point and let $\omega_{\mathrm{FS}}$ be the Fubini-Study metric. Let $L_p$ be the leaf of $p$, which, as we will see, is a hyperbolic Riemann surface, covered by the Poincaré disk $\Disk$. Let $\phi \colon \Disk \to L_p$ be a universal covering map such that $\phi(0) = p$ (unique up to composition by a rotation). For every $r \in (0,1)$, define
				\[\tau_r^p := \int_0^r\frac{\deriv t}{t}\int_{\Disk(0,t)}\phi^{*}\omega_{\mathrm{FS}},\]and
				\begin{equation}T_r^p(\omega) := \frac{1}{\tau_r}\int_0^r\frac{\deriv t}{t}\int_{\Disk(0,t)} \phi^{*}\omega.
				\label{eq:courant-moyenne}\end{equation}
				These definitions are independent of the choice of universal covering. Then we have
				\begin{thm}[Geometric Birkhoff theorem]
					The sequence $(T_r^p)_r$ converges to the current of theorem \ref{lethm} appropriately normalised, as $r \to 1$.
				\end{thm} Indeed, by a theorem of Fornæss-Sibony \cite[theorem 21]{FornaessSibonySurvey}, we have $\tau_r^p\xrightarrow[r\to 1]{} \infty$. By an elementary computation, we have
				\begin{align*}\abs{T_r^p\left(i\partial\overline{\partial}f\right)} 
				&=  \frac{1}{\tau_r^p}\abs{\int_0^{2\pi}f(\phi(re^{i\theta}))\deriv\theta - f(p)}\\
				&\leq \frac{2\norm{f}_\infty}{\tau_r^p},
				\end{align*}
				and so every limit point of $(T_r^p)$ is a harmonic current, which is positive and directed by $\mathscr{F}$. In fact, we will see also that we have the same statement (proposition \ref{prop:moyenne-poincare}) if we replace in equation \ref{eq:courant-moyenne} the normalising factor $\tau_r^p$ by
				\[\tilde{\tau}_r^p := \int_0^r\frac{\deriv t}{t}\int_{\Disk(0,t)}\phi^{*}\mathrm{vol}_{P},\]
				where $\mathrm{vol}_P$ is the volume form for the Poincaré metric on the leaves, changing however $T$ by a constant factor.

				\subsection{Outline of proof}
				The proof of theorem \ref{lethm} follows the approach of Deroin-Kleptsyn \cite{deroinkleptsyn}, which covers conformal and non-singular foliations. The difficulty in adapting their approach is mostly due to the singularities, with the non-conformality easier to handle.
				
				The main tool of the proof is to introduce a one-dimensional dynamical systems by studying Brownian motion on the leaves of the foliation. We will study the transverse dynamics along a Brownian path. In particular: do close leaves converge exponentially quickly along a typical Brownian path? We will work in a context where this is the case, and this contracting dynamics will be the key to the whole argument. The steps of the proof can be described as follows:
				
				\begin{itemize}
					\item The first step is to prove \emph{negativity of the Lyapunov exponents} for a path in $P$. That is, we need to prove that at the infinitesimal level, we have a contracting dynamics for a typical Brownian path in the plane $P$. First, there is a fundamental difficulty concerning even the existence of such Lyapunov exponents, which was resolved by the difficult work of Nguyên \cite{Nguyen}. Then, we use an observation of Candel \cite{CandelGarnett} and Deroin \cite{DeroinLevi} which relates the \emph{infinitesimal} behaviour of a typical Brownian path to curvature properties of normal bundles of leaves. This is perhaps not so surprising in that the curvature property of a normal bundle is related to the Laplacian of the logarithm of the norm of a section of the normal bundle, while Brownian motion is intimately related to the Laplacian. 
					\item The second step is to prove \emph{contraction}, that is, to improve the infinitesimal information of the first step to a local information: we need to show that a typical Brownian path in $P$ does not simply have exponential decay of its derivatives, but contracts a small transverse section exponentially. This is proved by discretising Brownian paths and using a simple contraction lemma concerning the iteration of maps with negative Lyapunov exponents.
					\item The third and most involved step is called \emph{similarity}, following Deroin-Kleptsyn. Given a path in $P$, starting from a regular point $p$, which is contracting in the sense of the preceding paragraph, we deduce information about a typical path starting from $q$, where $q$ is close enough to $p$. The idea is quite intuitive: Brownian motion behaves similarly from one leaf to another close one. But the technical realisation of this idea is quite involved. Here again we have to discretise Brownian paths, because the small scale behaviour of Brownian motion needs to be disregarded as there is no way it could be comparable in close leaves in general. Then we can reduce to comparing heat kernels on close leaves, and proving that certain products of heat kernels are close. See also the beginning of section \ref{sect:similarite} for a more detailed outline of this step.
				\end{itemize}
				Once these steps have been implemented, the end of the argument is not too difficult, and relies on the fact that we have tail events, essentially. By the result of Fornæss-Sibony \cite{FornaessSibony}, we know that a path in $P$ is distributed with respect to a certain harmonic measure. Then by the similarity, a path starting close to $P$ also has this property. Then a simple recurrence statement allows to obtain the result for every starting point, as the initial segment of the trajectory does not affect the distribution. Note that in everything above, our goal being \emph{unique} ergodicity rather than just ergodicity, it is crucial to handle \emph{all} starting points, not just almost all starting points.
				
				Here we have summarised essentially the work of Deroin-Kleptsyn, without explaining how the strategy is affected by the singularities, which, as we have said, are the essential difficulty. We describe the general idea now in a simple model case, that of a linear hyperbolic singularity. We will work in dimension $2$, as the phenomenon is already apparent in this context, but later on we will have to work in a more general context. So let us consider, for $\alpha, \beta \in \C^{*}$ with $\beta/\alpha \notin \R$, the vector field defined by
				\[X(x, y) = \alpha x \partial_x + \beta y \partial_y.\]
				for $(x, y) \in \Disk^2$, where $\Disk := \{z\in\C\,:\,\abs{z}<1\}$. We assume moreover that the imaginary part $\lambda$ of $-\beta/\alpha$ is positive, say. The flow of this vector field is given by $\phi^{\zeta}(x, y) = (xe^{\alpha\zeta}, ye^{\beta\zeta})$. We will consider the following leaf, which is simply given by $L = \{(x,0)\,:\,\abs{x}<1, x \neq 0\}$; note that it is a punctured disk. As we have said, we have a metric on our leaves, and we will be working with the hyperbolic or Poincaré metric, which in this case is
				\[\frac{\abs{\deriv z}}{\abs{z}\abs{\log\abs{z}}}.\]
				It might be possible to work with other metrics, but this one is quite natural because it comes naturally from the universal covering space of the leaf.
				
				Now take some $r \in(0,1)$, $x \in \C$ with $\abs{x} = r$. Consider $p = (x, 0)$. Consider the path $\gamma(t) = \left(e^{2i\pi t}x, 0\right)$ for $t\in\R$. Consider $T = \{(x,y)\,:\,\abs{y}<1\}$ a transverse section to $L$ at $p$. For any integer $N \geq 1$, the holonomy along the path $\gamma(t)$ with $0 \leq t \leq N$ is given by \[(x, y) \mapsto \left(x, \exp\left(2i\pi\frac{\beta}{\alpha}N\right)y\right).\]
				The hyperbolic length of the path is $2\pi N/\log(1/r)$, while the holonomy has a derivative at $0$ of modulus $\exp\left(2\pi N\lambda\right)$. So taking $r = e^{-N}$, we see that we can have paths of fixed length whose holonomy explodes, as we get closer to the singularity. The situation can even become potentially worse when considering paths which are not at a fixed distance from the singularity, as $\gamma$ was. What this computation does suggest, though, is that by taking paths of length about $1/\log(1/r)$, where $r$ is the distance to the singularity, we might still have control on the holonomy. We will see that this is the case in general (see lemma \ref{lem:controle-hol}). 
				
				Therefore, when we discretise Brownian paths, we must have smaller steps of discretisation when we approach the singularity. For this approach to work, we need to make sure that we do not spend too much time near the singularity, or at least that we do not approach the singularity too quickly, and this is exactly what the already mentioned fundamental work of Nguyên \cite{Nguyen} (see theorem \ref{thm:nguyen}) provides, by way of the ergodic theorem. This idea is also present in the work of Deroin-Dupont-Kleptsyn \cite{DDK}.
				\subsection{Organisation of the proof}
				In section \ref{sect:background}, we introduce the objects that we will use in the rest of the article, as well as some background information. Section \ref{sect:geometrie} focuses on some technical geometric constructions, in particular the Poincaré-type metric that we will use on leaves, projections of close leaves, necessary to compare the heat kernel, and control of the holonomy. Section \ref{sect:negativite} proves that the Lyapunov exponents are negative, then section \ref{sect:contraction} proves the contraction, and finally in section \ref{sect:similarite} we do the similarity step and in section \ref{sect:conclusion}, we give the recurrence argument that allows to conclude the proof of theorem \ref{lethm}. We also give some slight improvements of the results obtained.
			\subsection*{Acknowledgement}This work is the main part of my PhD thesis, written under the direction of Bertrand Deroin, to whom I am greatly indebted for his constant advice and support. I have also benefitted greatly from discussions with Christophe Dupont, and Victor Kleptsyn has helped me several times with different aspects of this problem. I also thank Jorge Vitório Pereira for giving me a useful reference on genericity, and Yann Chaubet for some useful explanations on PDEs.
				\section{Background}
				\label{sect:background}
				For all this section, we mostly follow Nguyên \cite{Nguyen, nguyen2018singular, NguyenSurvey}.
				\subsection{Notations}
				We write, for $r > 0$, $\Disk(r) := \{z \in \C\,:\,\abs{z}<r\}$.
				In all this article, we take $M$ a compact complex manifold of dimension $3$, and $P$ a surface in $M$; in practice we will take $M = \mathbb{P}^3(\C)$ and $P$ a projective plane in $\mathbb{P}^3(\C)$.
				
				It will be convenient to use the following $\log$-type function: let $\ell \colon (0, \infty) \to [1, \infty)$ be a smooth decreasing map such that for every $s \geq 0$, $\ell(s) \geq -\log s$ and for every $s \leq 1/3$, $\ell(s) = -\log s$. 
				We denote by $\deriv f_p$ or $\deriv_p f$ the differential of a map $f$ at a point $p$. For $T \geq 0$, we denote by $\mathrm{Leb}_{[0,T]}$ the restriction of the Lebesgue measure to $[0,T]$. We denote by $\norm{\cdot}_{\mathrm{C}^r}$ a $\mathrm{C}^r$-norm, with respect to an implicit choice of atlas. For a metric $g$, we denote by $\mathrm{vol}_g$ its volume form. We also write $\mathrm{diam}$ for the diameter of a set, and $\mathrm{dist}$ for the distance, when the metric used is clear from context (in the following it will be respective to an ambient metric on $M$ to be introduced later). Sometimes we write as an index the metric which is used.
				
				We will use the letters $C$, $C'$, $c$, etc. for positive constants which can change from line to line and are always allowed to depend on the foliation. To insist on this, and for clarity, we will often, though not always, write that they only depend "on the geometry of the foliation". We will also use the notation $A \lesssim B$ to mean that $A \leq CB$ for some constant $C > 0$ depending only on the geometry of the foliation, and similarly $A \gtrsim B$ to mean $B \lesssim A$ and $A \asymp B$ to mean both $A \lesssim B$ and $B\lesssim A$. Similarly the Landau notation will sometimes be used with constants that can depend on the geometry of the foliation. Sometimes we will also allow these implicit constants to depend on other parameters, which will be indicated by putting them as indices, e.g. $C_\eps$ for a constant which can depend on a parameter $\eps$.
				
				\subsection{Holomorphic foliations by curves} We consider on $M$ a (possibly singular) \emph{holomorphic foliation} by curves $\mathscr{F}$ with finite singular set $S$: this is the data of a covering of $M$ by open sets $\{U_i\}$, with a holomorphic vector field $X_i$ on each $U_i$ which has a finite number of zeroes, and such that for every $i, j$ with $U_i\cap U_j \neq \varnothing$, $X_i$ and $X_j$ differ by multiplication by a non-vanishing holomorphic function. The \emph{singular set} $S$ is the set of zeroes of the $X_i$; elements of $S$ are called \emph{singularities}. 
				
				\begin{notation}For any $p \in M \setminus S$, we write $\ell(p) := \ell(\dist(p,S))$. 
				\end{notation}
				
				It follows from the standard theory of analytic differential equations, see e.g. \cite{Ilyashenko,Abate} that $M \setminus S$ can be covered by a \emph{foliated atlas}, that is a cover of $M\setminus S$ by charts, called \emph{foliated charts}, $\{U_i\}_{i\in I}$ where for each $i \in I$, $U_i \simeq D_i \times T_i$, with $D_i$ a disc of $\C$ and $T_i$ a bi-disc of $\C^2$, and the change of coordinates between $(z, t) \in D_i \times T_i$ and $(w, s) \in U_j \simeq D_j \times T_j$ for some other such chart, is of the form:
				\begin{align*}
				w = w(z, t),\\
				s = s(t).
				\end{align*}
				
				Given a foliated chart as above, the sets defined by $t = \text{constant}$ are invariant by change of atlas. In a given foliated chart, a set of the form $t = \text{constant}$ is called a \emph{plaque}. A \emph{leaf} is a connected subset such that whenever it meets a plaque, it contains it entirely. Note that we use the convention that singularities are not leaves.
				
				For every \emph{regular} point $p \in M$, that is, $p \in M \setminus S$, denote by $L_p$ the leaf passing through $p$; and consider its tangent line. This defines a holomorphic line bundle $T_{\mathscr{F}}$ on $M\setminus S$; by Hartogs' phenomenon, it extends to a holomorphic line bundle on $M$. Note that $T_\mathscr{F}$ is a sub-bundle of the holomorphic tangent bundle $TM$ of $M$, and denote by $N_{\mathscr{F}} := TM/T_{\mathscr{F}}$ the \emph{normal bundle} of $\mathscr{F}$. 
				
				In all the article, we will use an ambient hermitian metric on $M$, denoted by $g_M$. We will later impose some additional hypotheses on $g_M$; but these are not very important, as the ambient metric plays not essential role. In particular, we will assume that we can choose charts such that near the singularities, the metric $g_M$ coincides with the standard hermitian metric on $\C^3$. Note that this metric induces hermitian metrics on $T_{\mathscr{F}}$, $N_{\mathscr{F}}$, etc. We will denote by $\dist$ the distance induced by this metric.
				
				For any $p \in P\setminus S$, we consider $S(p)$ the image of $\left(N_{\mathscr{F}}\right)_p \simeq \left(T{\mathscr{F}}\right)_p^{\perp}$ by the exponential map. For any $r > 0$, we also denote by $S(p,r)$ the set $\{\exp_p(w) \,:\,w \in \left(N_{\mathscr{F}}\right)_p, \abs{w} < r\dist(p,S)\}$. There is a constant $\newtarget{def:rho-section}{\rhoSection} > 0$ such that for $r < \rhoSection$ this is a transverse section to the foliation, and the distance induced on it is equal to the ambient distance on $M$.
				
				We make the following assumption in the rest of the article:
				\begin{assumption}The surface $P$ is invariant by $\mathscr{F}$, that is, each leaf intersecting $P$ is contained in $P$.
				\end{assumption}
				
				\subsection{Holonomy}The holonomy is a formalisation of the concept of Poincaré section. It allows to study the transverse dynamics of a foliation,  
				see Candel-Conlon \cite{CandelFoliationsI} for a precise definition. Given a path $\gamma \colon [0, 1] \to L$ between two points $p$ and $p'$ in a leaf $L$, we can consider two transverse sections $T$ and $T'$ to $L$ at $p$ and $p'$, respectively. The holonomy is obtained by sliding in plaques points of $T$ to $T'$.
				
				More precisely, assume first that there is a foliated chart $U \simeq D_i \times T_i$ such that $p, p' \in U$. Then, if $q \in T$ is close enough to $p$, the plaque through $q$ intersects $T'$ in a unique point, which is by definition the image of $q$ by the holonomy between $p$ and $p'$. For a general path $\gamma$ in a leaf, we can cut it into several pieces which are included in a foliation chart, choose intermediate transverse section, and compose the holonomy from all the different pieces. This gives the holonomy along $\gamma$, which we denote by $h_\gamma$. The germ of the holonomy map depends only on the homotopy class of the path with fixed endpoints.
				
				We will use variant of this notation: e.g. on a simply connected leaf, we will sometimes write $h_{pq}$ for the holonomy map along the geodesic between $p$ and $q$. If $\gamma$ is parametrised by an interval $I$ and $s, t \in \R$ such that $[s,t] \subset I$, we will also write $h_\gamma^{s,t}$ for the holonomy along the path $\restr{\gamma}{[s,t]}$. Finally, in the following, the transverse sections that we use will always be the $S(p,r)$ introduced above unless stated otherwise (in lemma \ref{lem:changer-transversale}).
				
				\subsection{Local model of hyperbolic singularities}\label{sect:local} Consider a singular point $p_0\in S$. We identify a small enough neighbourhood of $p$ with a small bidisc in $\C^2$ and $p_0$ with $0$. Consider a holomorphic vector field $X$ defining the foliation as above. We denote by $\{\phi^{\zeta}\}$ the local flow of $X$, so for $\abs{\zeta} \lesssim 1$ and $p$ close enough to $0$, we have $\phi^0(p) = p$ and \[\frac{\deriv \phi^\zeta}{\deriv\zeta}(p) = X(\phi^\zeta(p)).\] Call $p$ \emph{hyperbolic} if the eigenvalues of $\deriv X_p$ are all non-zero and no two of them are co-linear over $\R$ (this is an intrinsic property of the foliation defined by $X$). We will often reduce slightly the neighbourhoods so as to have estimates on the flow and its derivatives, with constants depending only on the geometry of the foliation.
				
				Call $p$ \emph{linearisable} if we can identify holomorphically a neighbourhood of $p$ with a neighbourhood of $0 \in \C^3$ such that $p$ is sent to $0$ and the vector field $X$ is linear, that is, of the form
				\[X(x, y ,z) = \alpha x \partial_x + \beta y \partial_y + \gamma z \partial z,\]
				for some complex numbers $\alpha, \beta, \gamma \in \C$.
				In general, this is not possible in dimension $3$, even in the hyperbolic case, see e.g. \cite{Ilyashenko}. In the hyperbolic case, however, it is always possible in dimension $2$, and we will freely use this fact. Note that the vector field $X$ preserves the plane $P$. Therefore, if $p \in P$, and in coordinates $(x, y, z)$ such that $P$ is given locally by $\{z= 0\}$, $X$ restricted to $P$ is given by
				\[X(x, y, 0) = \alpha x \partial_x + \beta y \partial_y.\]
				Then the flow $\phi^\zeta$ of $X$ is given by
				\[\phi^\zeta(x, y, 0) = (xe^{\alpha\zeta}, ye^{\beta\zeta}, 0),\]
				for every $\zeta \in \C$ small enough for the point to remain in the chosen neighbourhood of $p$.
				We also write
				\[N(x, y, 0) :=  \overline{\beta y}\partial_x -  \overline{\alpha x}\partial_y. \]
				This defines a normal vector to the foliation in the hermitian metric.
				
				\subsection{Poincaré metric on the leaves}
				Let $\mathscr{F}$ be a (possibly singular) holomorphic foliation on $M$, with finite set of singularities $S$. We say that a regular leaf $L$ is \emph{hyperbolic} if it is hyperbolic as a Riemann surface, that is, its universal covering space is biholomorphic to the Poincaré disk $\Disk$. The Poincaré disk has a hermitian metric of constant curvature $-1$, defined for every $\zeta \in \Disk$ by
				\[g_P(\zeta) := \frac{4\abs{\deriv\zeta}^2}{\left(1 -\abs{\zeta}^2\right)^2}.\] 
				For any hyperbolic leaf $L$, pushing forward the Poincaré metric on $\Disk$ to $L$ by any universal covering map gives a well-defined hermitian metric $L$ which we also call the \emph{Poincaré metric} on $L$, and we denote by $g_P$. For every $p \in M\setminus S$ such that $L_p$ is hyperbolic, we can relate the Poincaré metric to the ambient metric: there exists $\eta(x)$ such that $g_M(p) = \eta^2(p)g_P(p)$.
				\begin{defn}
					We say that the (possibly singular) foliation  by Riemann surfaces $\mathscr{F}$ is \emph{Brody hyperbolic} if every leaf is hyperbolic and the function $\eta$ is bounded.
				\end{defn} 
				The important fact about this metric is the following, which follows from the work of Lins Neto \cite{LinsNeto}, and e.g. Candel \cite{Candel} (see also Nguyên \cite[lemma 2.4]{Nguyen}):
				
				\begin{thm}[Canille Martins-Lins Neto \cite{CMLN}, Bacher \cite{Bacher}]Let $\mathscr{F}$ be a (possibly singular) holomorphic foliation on $M$, with finite set of singularities $S$. Assume that all the singularities are hyperbolic and that $\mathscr{F}$ is Brody hyperbolic. Then there exists a constant $C > 1$ such that for every $p \in M \setminus S$,
					\[C^{-1}s\ell(s) \leq \eta_P(p) \leq Cs\ell(s),\]
					where $s = \dist(p, S)$.
					\label{prop:typepoincare}
				\end{thm}
				We will work in a context where the Poincaré metric is continuous:
				\begin{thm}[Bacher \cite{Bacher}]
					Let $\mathscr{F}$ be a (possibly singular) holomorphic foliation on $M$, with finite set of singularities $S$. Assume that all the singularities are hyperbolic  and that $\mathscr{F}$ is Brody hyperbolic. Then $\eta$ is continuous.
				\end{thm}
				In fact Bacher gives Hölder-type estimates for the modulus of continuity, but we will not need them. For technical reasons, the Poincaré metric $g_P$ is not suitable for our purpose, as it is generally only known to be continuous in our context, and we will use smoothness in an essential way later on. So we will use a more general class of so-called \emph{Poincaré-type} metric, which are comparable to the Poincaré-metric but can be chosen to be smooth.
				\begin{defn}
					Let $g$ be a metric on the leaves in $M\setminus S$. We say that $g$ is a \emph{Poincaré-type metric} if $g$ is complete on every leaf and has uniformly bounded geometry and there exists a constant $C > 1$ such that for every $p \in M \setminus S$,
					\[C^{-1} \leq g/g_P(p) \leq C.\]
				\end{defn}
				See below \ref{subsect:poincaretype} for the precise definition of "uniformly bounded geometry" in this context. Note that the completeness of $g$ follows from the fact that it is within a bounded factor from the Poincaré metric.

				\subsection{Harmonic currents and foliation cycles} For all this subsection, useful references are Dinh-Nguyên-Sibony \cite{DNSHeat}, Ghys \cite{Ghys} and Fornæss-Sibony \cite{FornaessSibonySurvey}. Consider a complex manifold $M$ of arbitrary dimension in this subsection, with a possibly singular foliation $\mathscr{F}$, and denote by $S$ its set of singularities, which we assume is finite.
				
				Consider the space $\mathrm{C}^{\infty}_c(M, \mathscr{F})$ of continuous functions with compact support in $M \setminus S$, which are transversely continuous and tangentially smooth with respect to $\mathscr{F}$, namely, in foliated coordinates $(z, t) \in D \times T$ as above, all the derivatives with respect to $z$ exist and are continuous in $(z, t)$. Consider also the space $\mathrm{\Omega}^{1,1}_c(M, \mathscr{F})$ of compactly supported, transversely continuous and tangentially smooth with respect to $\mathscr{F}$, $(1,1)$-forms on each leaf. Call $\omega \in \Omega^{1,1}_c(M, \mathscr{F})$ \emph{positive} if, restricted to each leaf, it is a positive form in the usual sense, that is, for every $p \in L_p$, and $v \in \left(T_{\mathscr{F}}\right)_p$, $-i\omega(v, \overline{v}) > 0$. Given $f \in \mathrm{C}^{\infty}_c(M, \mathscr{F})$, we can naturally define $\partial\overline{\partial}f \in \Omega^{1,1}_c(P, \mathscr{F})$.
				
				The space $\Omega^{1,1}_c(M, \mathscr{F})$ is naturally a topological vector space. A \emph{ current} (of bidimension $(1, 1)$) directed by $\mathscr{F}$ is a continuous linear form $T$ on $\Omega^{1,1}_c(M, \mathscr{F})$. A \emph{harmonic current} directed by $\mathscr{F}$ is a current $T$ on this space which satifies $\partial \overline{\partial}T = 0$ in the weak sense, that is, for every $f \in \mathrm{C}^{1,1}_c(P, \mathscr{F})$, $T\left(\partial \overline{\partial}f\right) = 0$, and which is positive in the sense that for every positive $\omega \in \Omega^{1,1}_c(M, \mathscr{F})$, $T(\omega) > 0$. Note that we can extend such a current by zero on $S$, and it satisfies $\partial\overline{\partial} T = 0$.
				
				By the work of Berndtsson-Sibony \cite[theorem 1]{BS}, we know a priori that a harmonic current directed by $\mathscr{F}$ always exists for a singular holomorphic foliation by curves with a finite set of singularities.
				
				If $M$ is Kähler and has (complex) dimension $2$, we have moreover the following construction, due to Candel \cite{Candel} (see also Deroin \cite{DeroinLevi}). A harmonic current $T$ on $M$ with respect to $\mathscr{F}$ naturally defines a linear form on the $(1,1)$ Bott-Chern cohomology group $\mathrm{H}^{1,1}_{\text{BC}}(M,\R)$, which is isomorphic to $\mathrm{H}^1(M,\R)$ by the $\partial\overline{\partial}$-lemma, and therefore by duality, a class $[T]$. Given a line bundle $\mathscr{L}$ on $M$, we can define the intersection number $[T] \cdot c_1(\mathscr{L})$ to be $\frac{i}{2\pi}T(\Theta_\mathscr{L})$, where $\Theta_\mathscr{L}$ is the curvature of a hermitian metric on $\mathscr{L}$.
				
				A particular class of harmonic currents are the \emph{foliation cycles} of Sullivan \cite{Sullivan}, which are also called positive closed currents: these are the currents $T$ directed by $\mathscr{F}$ which are closed, i.e. for every $\eta$ which is a $1$-form on leaves, $T(\deriv \eta) = 0$, where we have defined forms on leaves and their differential in an analogous way as above, and which are positive in the same sense as above. If there is a compact leaf, integration on it naturally defines a foliated cycle, and more generally, by Ahlfors' lemma, any leaf which is covered by $\C$ gives a foliation cycle (see Candel \cite{Candel}). In particular, if there is no foliation cycle, all the leaves are hyperbolic, and in fact the foliation is Brody hyperbolic (see Fornæss-Sibony \cite[theorem 15]{FornaessSibonySurvey}) in the sense above:
				\begin{thm}
					Assume that $\mathscr{F}$ has no foliation cycles. Then $\mathscr{F}$ is Brody hyperbolic, and in particular all its leaves are hyperbolic Riemann surfaces.
				\end{thm}

				\subsection{Harmonic measures}For all this subsection, useful references are Candel-Conlon \cite{CandelFoliationsII}, Candel \cite{CandelGarnett} and Garnett \cite{Garnett}. We choose $g$ a hermitian metric on the leaves, for instance a Poincaré-type metric. Given a harmonic current $T$ with respect to $\mathscr{F}$ on $M$, we can consider an associated \emph{harmonic measure} $\nu$, which is defined by
				\[\int \varphi\deriv\nu = T\left(\varphi\mathrm{vol}_g\right)\]
				for any $\varphi \in \mathrm{C}^{\infty}_c(M, \mathscr{F})$. We will write $\nu = T \wedge \mathrm{vol}_g$ for short. For $g$ a Poincaré-type metric, we will see that this measure is finite under the assumption of theorem \ref{lethm}. Then, denoting by $\Delta = \Delta_g$ the leafwise Laplace-Beltrami operator associated to $g$, because we have $2i\partial\overline{\partial}\varphi = \Delta \varphi \mathrm{vol}_g$ on each leaf, we get that $\Delta \nu = 0$ weakly, i.e. for any $\varphi \in \mathrm{C}^{\infty}_c(M, \mathscr{F})$,
				\[\int \left(\Delta \varphi\right)\deriv\nu = 0.\] 
				More generally, any positive finite Borel measure $\nu$ which satisfies this property and gives no mass to the singularities and the non-hyperbolic leaves is called a harmonic measure.
				
				There is a correspondence between harmonic currents, which are more intrinsic as they do not depend on the choice of a metric, and harmonic measures, which is described by the following theorem:
				\begin{thm}Consider on $M$ a holomorphic foliation by curves $\mathscr{F}$ with a finite set of singularities, all of which are hyperbolic. Let $g$ be a Poincaré-type metric. Let $T$ be a positive harmonic current directed by $\mathscr{F}$.
					\begin{itemize}
						\item The current $T$ can be extended to a positive closed harmonic current on $M$ (see \cite{BerndtssonSibony} for a definition and proof),
						\item The relation $T\mapsto \nu = T \wedge \mathrm{vol}_g$ is a one-to-one correspondence between the convex cone of harmonic currents $T$ and the convex cone of harmonic measures $\nu$ (with respect to the Poincaré type metric $g$),
						\item The current $T$ is extremal in the convex cone of directed positive harmonic currents, if and only if $\nu =T\wedge\mathrm{vol}_g$ is \emph{ergodic}, that is, every Borel subset of $X$ which is saturated by leaves has full or zero measure with respect to $\nu$,
						\item Assume that $g$ is transversely continuous. Then each harmonic measure $\nu$ is $D_t$-invariant for every $t \geq 0$, i.e.
						\[\int D_t f\deriv\nu = \int f\deriv \nu,\]
						for every $f \in \Leb^1(\nu)$ (see below for the definition of the diffusion semi-group $(D_t)$).
					\end{itemize}
				\end{thm}
				\begin{rem}
					For proofs and references for this theorem, see Nguyên \cite[theorem 2.7]{Nguyen}. The result are stated for linearisable singularities but they also hold for merely hyperbolic singularities here given lemma \ref{prop:typepoincare}, see in particular \cite[lemma 4.3 and proof of proposition 4.2]{DNSHeat} and Bacher \cite[section 4]{BacherHeat}). The last part of the theorem follows from the work of Candel \cite{CandelGarnett} or Candel-Conlon \cite{CandelFoliationsII} (see also the remarks on the notions of "weakly harmonic", "quasi-harmonic", and "harmonic" measures in the work of Nguyên \cite[remark 2.13]{NguyenSurvey} \cite[remark 2.9]{NguyenOseledec}), or from combining \cite{BacherHeat} and \cite{DNSHeat}.
				\end{rem}
				
				Finally, we quote the main result of Nguyên \cite{Nguyen}. It is a difficult result which is an essential input for what follows.
				\begin{thm}[Nguyên \cite{Nguyen}]Consider a compact complex projective surface $S$ with a singular holomorphic foliation, with all its singularities hyperbolic and no foliation cycle directed by it. Then for any harmonic measure $\nu$ with respect to a Poincaré-type metric on the leaves, we have
					\[\int\ell(p)\deriv\nu(p) < \infty.\]
					\label{thm:nguyen}
				\end{thm}
				
				\subsection{Brownian motion on the leaves} Now we briefly describe the Brownian motion on the leaves, introduced by Garnett \cite{Garnett}. For more information, see Candel-Conlon \cite[chapter 2]{CandelFoliationsII}, Candel \cite{CandelGarnett}, and Nguyên \cite[subsection 2.4]{Nguyen}. Here we follow mostly Deroin-Kleptsyn \cite{deroinkleptsyn}. We choose a complete hermitian metric on the leaves with uniformly bounded geometry.
				For every $p \in M \setminus S$, consider on $L_p$ the heat equation:
				\[\frac{\partial u}{\partial t} = \Delta u,\]
				\[u(t, \cdot) \xrightarrow[t\to 0]{} f,\]
				in the sense of distributions,
				where $f$ is the initial data. Because $g$ is complete with bounded geometry, there is a unique solution defined for all non-negative time, which we denote $(t, q) \mapsto D_tf(q)$, where $(D_t)$ is called the \emph{diffusion semigroup}. The diffusion $D_t$ is given by convolution with the \emph{heat kernel} associated to $g$, which we denote by $p_g = p_g(x, y; t)$:
				\[D_tf(x) = \int f(y)p_g(x, y; t)\deriv\mathrm{vol}_g(y),\]
				for $f$ Borel measurable on $M\setminus S$. It satisfies for every $s, t\geq 0$, $D_0 = \mathrm{id}$, $D_t 1 = 1$ and $D_s \circ D_t = D_{s+t}$. This semi-group has the following property:
				\begin{thm}[Garnett \cite{Garnett}]
					For every $f \in \mathrm{C}_c(M\setminus S)$, the function $D_tf$ is continuous.
				\end{thm}
				
				Associated to the diffusion semi-group is the Brownian motion along the leaves of the foliation, a Markov process with transition kernel $p(x, y; t)$. We denote by $\Gamma$ the set of continuous paths $\omega \colon [0, \infty) \to M \setminus S$ which remain in the same leaf, by $\Gamma_p$ the paths in $\Gamma$ starting at $p$. We can consider the $\sigma$-algebra generated by cylinder sets, that is sets $C$ for which there exists a finite set $F$ and Borel sets $(B_t)_{t \in F}$ in $M \setminus S$ such that $C = \{\gamma \,:\, \gamma(t) \in B_t \text{ for every } t \in F\}$. 
				We also consider, for every $p \in M \setminus S$ and $\tilde{p}$ a lift of $p$ in $\tilde{L}_p$, the set $\tilde{\Gamma}_{\tilde{p}}$ of continuous paths $[0, \infty) \to \tilde{L_p}$, and the disjoint union $\tilde{\Gamma}$ of the $\tilde{\Gamma}_{\tilde{p}}$. On $\tilde{\Gamma}$ we can similarly define a $\sigma$-algebra of cylinder sets, see Nguyên \cite[section 2.4]{Nguyen} for details: the union of the universal covering spaces of the leaves give a foliation $\tilde{\mathscr{F}}$ on a manifold $\tilde{M}$, and we can then adapt the definition of cylinder sets.
				
				On the measurable spaces $\tilde{\Gamma}_p$ we can define the \emph{Wiener measure} $\tilde{W}_{\tilde{p}}$: given $n \geq 1$ and a finite sequence $t_1 < t_2 < \ldots < t_n$ of times and Borel subsets $B_1, B_2, \ldots, B_n$ of $\tilde{M}$, we have
				\begin{align*}\tilde{W}_{\tilde{p}}\left\{\tilde{\gamma} \in \tilde{\Gamma}_{\tilde{p}}\,:\text{ for }i = 1, 2, \ldots, n,\,\tilde{\gamma}(t_i) \in B_i\right\}
				= \left(D_{t_1}\left(\chi_{B_1}D_{t_2 - t_1}\left(\chi_{B_2}\cdots D_{t_n-t_{n- 1}}(\chi_{B_n})\cdots\right)\right)\right)(p)\end{align*}
				where $\chi_{B_i}$ is the indicator function of $B_i$ (here the diffusion operators are with respect to the lift of the metric $g$ on $\tilde{L}_p$),
				We also have the Wiener measure on $\Gamma_p$, defined by $W_p(A) = \tilde{W}_{\tilde{p}}(\pi_p^{-1}A)$ for any measurable set $A$, where $\pi_p \colon \tilde{L}_p \to L_p$ is the projection. Note that it is independent of the choice of lift of $p$. Finally note that we have for every $t \geq 0$ and bounded Borel measurable function $\varphi$ on $\tilde{M}$ and $\psi$ on $M$, $q \in \tilde{L}_p$:
				\[\int \varphi(\gamma(t))\deriv\tilde{W}_{\tilde{p}}(\gamma) = D_tf(\tilde{p}),\]
				\[D_t(\psi \circ \pi_p)(q) = D_t\psi(\pi_p(q)).\]
				
				Given a harmonic measure $\nu$ on $P$ or $M$, we define the measure $\overline{\nu}$ on $\Gamma$ as follows
				\[\overline{\nu}(B) = \int\deriv\nu(p)W_p(\Gamma_p \cap B)\]
				where $B$ is a measurable set. Given $t \geq 0$ and $\gamma \in \Gamma$, we define $\sigma_t(\gamma)$ as the shifted path $s \mapsto \gamma(t + s)$. Then we have the important property that if $\nu$ is harmonic and $\varphi \in \Leb^1(\overline{\nu})$, we have for every $t \geq 0$
				\[\int\varphi \circ \sigma_t\deriv\overline{\nu} = \int\varphi\deriv\overline{\nu}.\]
				In fact, we have the following \emph{random ergodic theorem} (see Candel \cite{CandelGarnett} for instance):
				\begin{thm}
					If the measure $\nu$ is an ergodic harmonic measure,	the system $(\Gamma, (\sigma_t)_{t \geq 0}, \overline{\nu})$ is ergodic.
					In fact, for every $t > 0$, $(\Gamma, \sigma_t, \overline{\nu})$ is ergodic.
				\end{thm}
				We will also need repeatedly the following lemma: consider a measurable set $A$ which is \emph{tail-type}, that is, only depends on asymptotic properties of a path (see Candel \cite[section 6]{CandelGarnett} for a precise definition):
				\begin{lem}
					For a tail-type event $A$, assume that $W_p(A \cap \Gamma_p) = 1$. Then the same is true for any other point in the same leaf as $p$.
					\label{lem:queue}
				\end{lem}
				In general, the function $p \mapsto W_p(A \cap \Gamma_p)$ is harmonic along leaves, as can be seen using the Markov property (see Candel \cite[section 6]{CandelGarnett}) for a tail event $A$, and then the result follows from the maximum principle. 
				
				The following lemma is a consequence of standard bounds, for the heat kernel of manifolds with bounded geometry (see Nguyên \cite[section 4]{Nguyen}).
				\begin{lem}
                \label{lem:saut-brownien}
					There exists a constant $\RLemSautBrownien > 0$ and $\newtarget{lem:saut-brownien}{\cstLemSautBrownien} > 0$ such that for every 
					$\delta \in (0,2)$, $p \in M \setminus S$ and $\tilde{p}$ a lift in $\tilde{L}_p$, and all $R \geq \RLemSautBrownien$, we have
					\[\tilde{W}_{\tilde{p}}\left(\gamma \in \tilde{\Gamma}_{\tilde{p}} \,:\, \sup_{0 \leq t \leq \delta}\dist_g(\gamma(0), \gamma(t)) \geq R\right) \leq s^{-3/2}\exp\left(-\cstLemSautBrownien\frac{R^2}{\delta}\right).\]
					
				\end{lem}

				\section{Geometry of the foliation}
				\label{sect:geometrie}
				This section is devoted to technical statements concerning the geometry of the foliation. The proofs are somewhat tedious, and are often quite intuitive but require some care. They are very close in spirit to those of Dinh-Nguyên-Sibony \cite{DNSII}.
				\subsection{Construction of a smooth Poincaré-type metric}
				\label{subsect:poincaretype}
				\begin{defn}
					Let $g$ be a hermitian metric on a Riemann surface $Y$, not necessarily compact. We say that $g$ has \emph{bounded geometry} if we can cover $Y$ by a countable set of local charts such that there exists $\rho > 0$ and for every integer $r \geq 0$, a constant $M_r > 0$ with:
					\begin{itemize}
						\item The domain of the chart contains a ball of radius $\rho$,
						\item Identifying $g$ with $g(z)\abs{dz}$ in the chart, all the derivatives of $g$ and $g^{-1}$ up to order $r$ are bounded above by $M_r$.
					\end{itemize}
					
					Given a family of such metrics on a family of Riemann surfaces, we say that it has \emph{uniformly bounded geometry} if we can choose $\rho$ and the $(M_r)_{r \geq 0}$ uniformly.
				\end{defn}
				
				\begin{rem}
					In particular there exists $\rho_0 > 0$ depending only on $\rho$ and $M_0$ such that the injectivity radius of $Y$ is larger than $\rho_0$ at every point, and the curvature of $Y$ is bounded from above and below by constants depending only on $M_2$.
				\end{rem}
				\begin{lem}
					There exists a smooth Poincaré-type metric $g$ on the leaves which has uniformly bounded geometry on each leaf.
					\label{lem:geombornee}
				\end{lem}
				
				\begin{proof}[Proof of lemma \ref{lem:geombornee}]
					The proof is close to that of a similar result of Lins Neto \cite{LinsNeto}, but simpler. We only need to prove this in a neighbourhood of a singular point. Consider a local chart identified with $\mathbb{D}^3$, where $0$ corresponds to the unique singular point in this neighbourhood. Write $\phi^\zeta(p)$ for the local flow of a defining vector field, with $\zeta$ the complex time. We assume, as we may, that the ambient hermitian metric corresponds to the standard hermitian metric in this chart. Consider then for $p \in \mathbb{D}^2$ :
					\[g_p(V) = \frac{1}{\norm{X(p)}^2\ell(p)^2}\norm{V}^2\]
					
					Let $p \in \mathbb{D}^3$ with $\norm{p}$ small enough. Let $c > 0$ be small enough. Consider the map defined on $\mathbb{D}(c)$ by $t \mapsto \phi_p(\ell(p)t)$. It defines a local chart on $L_p$ around $p$.
					In this chart, the metric can be written, for $t \in \mathbb{D}(c)$ :
					\begin{align*}\left(f_p^*g\right)_t(v) = \frac{\ell(p)^2}{\ell(\phi^\zeta(p)))^2}\abs{v}^2
					\end{align*}
					where we have written $\zeta = \ell(p)t$. Reducing slightly the neighbourhood if necessary, $\norm{p}e^{C\abs{\zeta}} \geq \norm{\phi^\zeta(p)} \geq \norm{p}e^{-C\abs{\zeta}}$ for a constant $C > 0$, we deduce that $\left(f_p^*g\right)(v) \asymp \norm{v}^2$. Moreover, for any integer $r \geq 1$, we can check that  there exists $A_r > 0$ such the derivatives of order at most $r$ of $\zeta \mapsto 1/\log\left(\norm{\phi_p(\zeta)}\right)$ are bounded in modulus by \[\frac{A_r}{\log\left(\norm{\phi^\zeta(p)}\right)^{r + 1}},\] for $\abs{\zeta}$ small enough. This implies that there exists $C_r > 0$ such that the derivatives of order at most $r$ of $t \mapsto \left(f_p^*g\right)(v)$ are bounded in modulus by $C_r$. So $g$ has uniformly bounded geometry in a neighbourhood of the singular point.
				\end{proof}
				We note that in the notations of the proof, $\{\phi^\zeta(p)\,:\,\abs{\zeta}\lesssim 1\}$ contains, and is contained in, a ball for $g$ around $p$ of radius $\asymp \ell(p)^{-1}$.
                \begin{notation}We fix $g$ such a metric and denote by $\dist_g$ the distance induced by the Poincaré metric $g$ on a leaf. We will also write $B_g$ for the balls for this distance. From now on, when we say that a constant "depends only on the geometry of the foliation", we allow it to depend on $g$.
				\end{notation}
                Note also:
				\begin{lem}
					There exist constants $\eps > 0$, $C > 0$ such that if $p, q \in M \setminus S$ are in the same leaf and $\dist_g(p, q) \leq \eps\ell(p)^{-1}$, then $\dist(q, S) \geq \exp\left(-C\dist_g(p,q)\ell(p)\right)\dist(p, S)$ and $\abs{\ell(q) - \ell(p)} \lesssim \ell(p)\dist_g(p, q)$.
					\label{lem:dist-sing}
				\end{lem}
				\begin{proof}
					It suffices to assume that $p$ is close enough to a singularity, by a compactness argument and choosing $\eps$ small enough to ensure that $q$ is close to the singularity, too. Then we work in the local model to prove both estimates.
				\end{proof}
				\subsection{Projections along a path. } In this subsection, we record some abstract lemmas that will be used to define projections between leaves. We leave the proofs to the reader. They are similar in spirit to the construction of the holonomy (see Candel-Conlon \cite{CandelFoliationsII}).
				Let $x, y \in M$ be two regular points. Let $L$ be the leaf passing through $y$. We say that $y$ is an \emph{orthogonal projection} of $x$ if there exists $w \in \left(T_yL\right)^{\perp}$ such that $\exp_y\left(w\right) = x$.
				
				Let $U$ and $W$ be open sets insides two leaves. We call a map $\pi \colon U \to W$ a \emph{local projection from $U$ to $W$} if
				\begin{itemize}
					\item the map $\pi$ is a diffeomorphism onto its image,
					\item for every $x \in U$, $\pi(x)$ is an orthogonal projection of $x$ and moreover it is the only such point in $W$.
				\end{itemize}
				
				\begin{lem}[Unique continuation]Let $W$ and $U_1, U_2$ be open sets inside leaves. Let $\pi_1$ and $\pi_2$ be local projections from $U_1$ and $U_2$ respectively to $W$. Then $\pi_1$ and $\pi_2$ coincide on $U_1 \cap U_2$ and glue to define a map $\pi$ on $U_1 \cup U_2$ which is a local projection from $U_1 \cup U_2$ to $W$.
					\label{lem:proj-unique-continuation}
				\end{lem}
				\begin{defn}Let $\gamma \colon [0, 1] \to M$ be a piecewise $\mathrm{C}^1$ path inside a leaf $L$, and $q$ an orthgonal projection of $\gamma(0)$. Let $L'$ be the leaf passing through $q$. We say that $q' \in L'$ is a \emph{projection along $\gamma$ starting from $q$} if there exist $n \in \N$, $0 = t_0 < t_1 < \ldots < t_n = 1$, a sequence $\left(U_i\right)_{0 \leq i \leq n}$ of open sets in $L$, a sequence $\left(W_i\right)_{0 \leq i \leq n}$ of open sets in $L'$, and for $i = 0, \ldots, n$, a local projection $\pi_i$ from $U_i$ to $W$ such that $\pi_n(\gamma(1)) = q'$ and for every $i = 0, \ldots, n - 1$, $\gamma([t_i, t_{i + 1}]) \subset U_i$ and $\pi_{i + 1}(U_{i + 1}) \subset W_i$.
					\label{def:proj-locale-chemin}
				\end{defn}
				
				Lemma \ref{lem:proj-unique-continuation} implies the following lemmas by standard arguments:
				\begin{lem}
					Let $\gamma \colon [0, 1] \to M$ be a path inside a leaf $L$, and $q$ an orthgonal projection of $\gamma(0)$. Then there exists at most one point $q'$ which is an orthogonal projection along $\gamma$ starting from $q$.
					\label{lem:proj-bien-definie}
				\end{lem}
				\begin{lem}
					Let $p, p', p''$ be points in a leaf $L$. Let $\gamma_1, \gamma_2$ be two paths in $L$, from $p$ to $p'$ and from $p'$ to $p''$ respectively. Let $\gamma$ be the composition of $\gamma_1$ and $\gamma_2$.
					Let $q$ be an orthogonal projection of $p$, $q'$ a projection of $p'$ along $\gamma_1$ starting from $p$, and $q''$ a projection of $p''$ along $\gamma_2$ starting from $q'$.
					
					Then $q''$ is a projection of $p''$ along $\gamma$ starting from $p$.
				\end{lem}
				\begin{lem}[Homotopy invariance]Let $p$ and $p'$ be points in a leaf $L$. Consider $q$ an orthogonal projection of $p$. Suppose that we are given an homotopy $(\gamma_s)_{0 \leq s \leq 1}$ with fixed endpoints between paths from $p$ to $p'$. Assume that for every $s \in [0, 1]$, there is a point $q'_s$ which is an orthogonal projection of $p'$ along $\gamma_s$ starting from $q$. Then $q'_s$ is independent of $s$.
					\label{lem:proj-invariance-homotopie}\end{lem}
				\subsection{Projections of leaves. }
				We will prove below (subsection \ref{sub:etude-locale}) the following lemma, which constructs local orthogonal projections. In this subsection we apply it in conjunction with the lemmas of the previous section to define "global" orthogonal projections.
				\begin{lem}
                					\label{lem:proj-locale}
					For every $\lambdaLemProjLocale > 0$, there exist constants  $\newtarget{lem:proj-locale}{\epsLemProjLocale} \in (0, 1/2)$ and $\rLemProjLocale > 0$, $\RLemProjLocale > 0$, $\kappaLemProjLocale > 0$ such that for every $p \in P \setminus S$, $q \in M \setminus S$ with $q \in S(p, \epsLemProjLocale \dist(p, S))$, there exists a local projection $\pi$ from the ball $B\left(q, \frac{\rLemProjLocale}{\ell(p)}\right)$ to $B\left(q, \frac{\lambda \RLemProjLocale}{\ell(p)}\right)$ whose image is included in $B\left(q, \frac{\RLemProjLocale}{\ell(p)}\right)$.
					
					Moreover, for every $z \in B\left(q, \frac{\rLemProjLocale}{\ell(s)}\right)$, we have
					\begin{equation}\dist(\pi(z), z) \leq e^\kappaLemProjLocale \dist(p, q).
					\label{eq:dist-projA}\end{equation}
					
					Here the balls are with respect to the Poincaré-type metric on the leaves.

				\end{lem}
				Here the role of the parameter $\lambda$, which will be chosen a constant depending only on the geometry of the foliation, is to have compatibility on overlaps as in the lemmas of the previous section.
				
				\begin{lem}\label{lem:proj-globale}
                There exists a constant $\newtarget{lem:proj-globale}{\CstLemProjGlobale} \geq 1$ and $\eps_0 > 0$ such that the following holds:
                
					Let $p \in P \setminus S$ with lift $\tilde{p} \in \tilde{L}_{\tilde{p}}$, $q\in M\setminus S$ be two points, with $q \in S(p)$ with lift $\tilde{q} \in \tilde{L}_q$ and
					\[\exp\left(\CstLemProjGlobale\ell(p)e^{\CstLemProjGlobale R}\right)\dist(p,q) \leq \eps_0.\]
					Then there exists a smooth map $\Phi_{\tilde{q}\tilde{p}} \colon B_g(\tilde{q},R) \to \tilde{L}$ sending $\tilde{q}$ to $\tilde{p}$ which is locally an orthogonal projection, i.e. for every $z \in B_g(\tilde{q},R)$, the  map $\Phi_{\tilde{q}\tilde{p}}$ coincides locally around $z$ with the local projection of lemma \ref{lem:proj-locale} from $z$ to $\Phi_{\tilde{q}\tilde{p}}(z)$. This map $\Phi_{\tilde{q}\tilde{p}}$ is a diffeomorphism onto its image, which is $2$-Lipschitz and whose inverse is $2$-Lipschitz. Moreover the image of $\Phi_{\tilde{q}\tilde{p}}(z)$ is contained in a ball (for the Poincaré-type metric) of radius $2R$, and contains a ball of radius $R/2$.
					
					Moreover, for every $z \in B_g^L(\tilde{q}, R)$ and every $R'$ such that
					\[\exp\left(\CstLemProjGlobale\ell(w)e^{\CstLemProjGlobale R'}\right)\dist(z, w) \leq \eps_0,\]where $w = \Phi_{\tilde{q}\tilde{p}}(z)$,
					the map $\Phi_{zw} \colon B_g^L(z, R') \to L$ coincides with $\Phi_{\tilde{q}\tilde{p}}$ on $B_g(q, R) \cap  B_g(z, R')$.
					
				\end{lem}
				\begin{proof}
					We start by proving the following claim:
					\begin{claim}There exists a constant $C \geq 1$ such that the following holds:
						Let $\gamma$ be a piecewise geodesic path between two points $q, q' \in M \setminus S$ in the same leaf of length $R$. Let $p \in M \setminus S$ be such that $p$ is an orthogonal projection of $q$. Assume that $\exp\left(C\ell(p)e^{CR}\right)\dist(p, q) \leq \epsLemProjLocale$. Then there exists a point $p'$ in the leaf of $p$ which is a projection of $q'$ along $\gamma$ starting from $p$, and moreover:
						\[\ell(p') \leq e^{CR}\ell(p),\] 
						\[\dist(q', p') \leq \exp\left(C\ell(p)e^{CR}\right)\dist(q, p).\]
					\end{claim}
					\begin{proof}
						We prove that there exists a constant $d_0 > 0$ such that the claim holds if $R \leq d_0$. Then the general case follows by induction. Thus we can assume that $\gamma$ is a minimising unit-speed geodesic.
						
						If $R\ell(p)$ is small enough, the claim follows by lemmas \ref{lem:dist-sing} and \ref{lem:proj-locale}. Otherwise, we construct inductively a finite sequence of times $t_0, t_1, \ldots, t_N$, points $p_0, p_1, \ldots, p_N$ and $q_0, q_1, \ldots, q_N$, and maps $\pi_0, \pi_1, \ldots, \pi_N$, for some $N \in \N$. We start with $t_0 = 0$, $p_0 = p$, $q_0 = q$ and $\pi_0$ the local projection given by lemma \ref{lem:proj-locale} associated to $p_0$ and $q_0$ (this is possible since the condition on $p$ and $q$ implies that $\dist(p,q) \leq \epsLemProjLocale\dist(p, S)$).
						Assume that we have constructed the first $n$ terms of the sequence, for some positive integer $n$. Let $t_{n} = t_{n - 1} + \eta\ell(p_{n - 1})^{-1}$ for $\eta > 0$ a small enough constant, $q_{n}= \gamma(t_{n})$, $p_{n} = \pi_{n - 1}(q_{n})$ which is well-defined. By induction, equation \ref{eq:dist-projA} implies
						\begin{equation}\dist(p_n, q_n) \leq e^{n\kappaLemProjLocale}\dist(p, q).
						\label{eq:lem-proj-chemin}\end{equation}

						Then by lemma \ref{lem:dist-sing}, if $\eta$ is small enough $ \dist(p_n,S) \geq e^{-1}\dist(p_{n -1}, S)$ and $\ell(p_n) \leq \ell(p_{n -1}) + 1$. Therefore by induction, $\dist(p_n, S) \geq e^{-n}\dist(p, S)$ and $\ell(p_n) \leq \ell(p) + n$, and thus, so long as $n \leq \ell(p)$, $\dist(p_n, S) \geq e^{-\ell(p)}\dist(p, S)$ and $\ell(p_n) \leq 2\ell(p)$. Therefore
						\[\frac{\dist(p_n, q_n)}{\dist(p_n, S)} \leq e^{n(1+\kappaLemProjLocale )}\frac{\dist(p,q)}{\dist(p, S)} \leq e^{C\ell(p)}\dist(p,q),\] for some constant $C > 0$. Choose the constants in the claim in such a way that this implies $\dist(p_n, q_n) \leq \epsLemProjLocale\dist(p_n, S)$. Then we can consider the local projection given by lemma \ref{lem:proj-locale} associated to $p_n$ and $q_n$. 
						We continue this way until $n \geq \frac{1}{2}\ell(p)$. Then the data we have constructed gives a projection along a path as in definition \ref{def:proj-locale-chemin}. 
						This establishes the claim since by induction $t_n \geq \frac{n}{2}\eta\ell(p)^{-1} \geq \frac{1}{4}\eta$ which is a constant depending only on the geometry of the foliation. 
					\end{proof}
					
					Given the claim, using the lemmas of the previous paragraph (lemmas \ref{lem:proj-bien-definie} and \ref{lem:proj-invariance-homotopie}), we obtain the existence of the map $\Phi_{\tilde{q}\tilde{p}}$ which, by construction, is locally an orthogonal projection, which is smooth because it coincides locally with the orthogonal projection of lemma \ref{lem:proj-locale} — here we use the fact that the radius of injectivity of the leaves with the Poincaré-type metric is uniformly bounded away from zero.
					
					We prove the last part of the lemma. Choose $z' \in B_{g}(z, R')$. Consider the geodesic $\gamma_1$ from $\tilde{q}$ to $z'$, and the path $\gamma_2$ consisting of the geodesic from $\tilde{q}$ to $z$ and then the geodesic from $z$ to $z'$. 
					Consider the projection obtained by the claim using the geodesic from $\tilde{q}$ to $z$, and then that obtained by the claim from the geodesic between $z$ and $z'$. Concatenating the associated data as in definition \ref{lem:implicit} using lemma \ref{lem:proj-unique-continuation}, we obtaine a projection of $\tilde{q}$  along $\gamma_1$ starting from $p$. Using the claim applied to $\gamma_2$, we obtain a projection of $\tilde{q}$ along $\gamma_2$ starting from $p$. For an appropriate value of $C$, using the lemma \ref{lem:proj-invariance-homotopie} and the elementary lemma below, we obtain the compatibility condition on $\Phi$. 
				\end{proof}
				
				\begin{lem}
					There exists a constant $K > 0$ such that the following holds:
					
					Let $a, b, c$ be three points in a leaf with the Poincaré-type metric  (recall that the Poincaré-type metric is comparable to a hyperbolic metric).
					Consider $\gamma$ the geodesic from $a$ to $c$, and $\gamma'$ the broken geodesic from $a$ to $b$ and then $b$ to $c$.
					Then there exists a homotopy with fixed endpoints between $\gamma$ and $\gamma'$ through piecewise geodesic paths of length at most $K\left(\dist(a, b) + \dist(b, c) + 1\right)$.
					
				\end{lem}
				
				We isolate the following straightforward consequence of the proof of lemma \ref{lem:proj-globale} (see also Nguyên \cite[proposition 3.3]{Nguyen}):
				\begin{lem}
					There exists a constant $C > 0$ such that for every continuous path $\gamma$ in $M \setminus S$, with lift $\tilde{\gamma}$ to its universal cover, and every $t \geq 0$, $\ell(\gamma(t)) \leq \ell(\gamma(0))\exp\left(C\dist(\tilde{\gamma}(0),\tilde{\gamma}(t)\right)$.
					\label{lem:ell-sing}
				\end{lem}
				\subsection{Local study near a singularity. }In this subsection, we prove technical estimates on local projections near a singularity, as well as the essential control of the holonomy on small steps. The part about projections is much easier if we assume that we have linearisable singularities, as it then simply follows from an homogeneity argument and compactness (see Dinh-Nguyên-Sibony \cite{DNSII}). The control of the holonomy also follows from such an argument in the  analogous setting of dimension $2$ (see Dupont-Deroin-Kleptsyn \cite{DDK}). We therefore advise the reader to skip the proofs. Note that similar arguments were also used by Bacher \cite{Bacher} in his study of the Poincaré metric in the presence of non-linearisable singularities.
				
				We use the notation of subsection \ref{sect:local}. We take $p \in P\setminus S$.
				\label{sub:etude-locale}

				Let $F = F_{p} \colon \Disk(c)^4\to \mathbb{C}$ for some small $c > 0$ be the function defined by
				\[F(t, u, \zeta, \xi) := \frac{1}{\norm{p}^2}\det\left(\phi^{\zeta}(q) - \phi^{\zeta + \xi}(p), N\left(\phi^{\zeta + \xi}(p)\right), \frac{\partial}{\partial z}\right)\]
				for $(t, u, \zeta, \xi) \in \Disk(c)^4$, where we write $q = p + tN(p) + u\norm{p}\frac{\partial}{\partial z}$ as a short-hand here and below; note that $\norm{q - p} \asymp \left(\abs{t} + \abs{u}\right)\norm{p}$. The map $F$ is smooth in all variables; it is holomorphic with respect to $t$ and $u$ but not $\zeta$ and $\xi$ since the map $N$ is anti-holomorphic. We will apply to $F$ the following simple quantitative form of the implicit function theorem (see for instance Barreira-Pesin \cite[lemma 7.5.2]{BarreiraPesin}).
				\begin{lem}[Quantitative implicit function theorem]Consider a $C^2$ map\begin{align*}F &\colon B_{\R^m}(0, 1) \times B_{\R^n}(0, 1) \to \R^n\\ &(x, y) \mapsto F(x, y).\end{align*}Assume that $\deriv_yF(0, 0)$, the partial differential of $F$ with respect to the variable $y$, is invertible. Let $C > 0$ be such that we have the following bounds on the partial differentials: \[\norm{\deriv_xF}_\infty, \norm{\deriv^2_{yy}F}_\infty, \norm{\deriv^2_{xy}F}_\infty \leq C \text{ and }\norm{\deriv_yF(0, 0)^{-1}}_\infty \leq C.\]
					Then there exist $r_0 = r_0(C) \in (0, 1)$, $r_1 = r_1(C) \in (0, 1)$ \emph{depending only on $C$} and a $C^1$ map $g \colon B_{\R^m}(0, r_0) \to B_{\R^n}(0, r_1)$ such that
					for every $x \in B_{\R^m}(0, r_0)$, $F(x, g(x)) = 0$, and $g(x)$ is the only $y \in B_{\R^n}(0, r_1)$ such that $F(x, y) = 0$. 
					
					Moreover, for every $\lambda \geq 1$, $g\left(B\left(0, \frac{r_0}{\lambda}\right)\right) \subset B\left(0, \frac{r_1}{\lambda}\right)$.
					\label{lem:implicit}
				\end{lem}
				First, it is easy to check that one can bound the norms of
				$\deriv_\zeta F, \deriv_t F, \deriv_u F$
				and
				\[\deriv^2_{\xi\xi}F, \deriv^2_{\zeta\xi}F, \deriv^2_{t\xi}F, \deriv^2_{u\xi}F\]
				by constants. Actually, it will be useful later on to note that $\norm{F}_{\mathrm{C}^r} \lesssim_r 1$ for every non-negative integer $r$.
				
				Now we compute the differential $\deriv_\xi F$ of $F$ with respect to $\xi$:
				\begin{align*}
				\frac{\partial F}{\partial \xi} 
				&= -\frac{1}{\norm{p}^2}\det\left(X\left(\phi^{\zeta + \xi}(p)\right), N\left(\phi^{\zeta +  \xi}(p)\right),\frac{\partial}{\partial z}\right) \\
				\frac{\partial F}{\partial \overline{\xi}} 
				&= \frac{1}{\norm{p}^2}\det\left(\phi^{\zeta}(q) - \phi^{\zeta + \xi}(p), \left(\deriv N\right)_{\phi^{\zeta+ \xi}(p)}\left(X\left(\phi^{\zeta + \xi}(p)\right)\right),\frac{\partial}{\partial z}\right).\end{align*}
				
				Therefore
				\begin{align*}
				\frac{\partial F}{\partial \xi}(0, 0, 0, 0) = -\frac{\norm{X(p)}^2}{\norm{p}^2},
				\end{align*}
				while $\left(\partial F/\partial \overline{\xi}\right)(0, 0, 0, 0) = 0$. This implies that the differential $\partial_\xi F$ of $F$ with respect to $\xi$ is an invertible map, and $\norm{\partial_\xi F(0, 0, 0, 0)^{-1}}$ is bounded above by a constant. 
				Therefore, applying lemma \ref{lem:implicit}, we obtain constants $r_0, r_1 > 0$ and a map $\xi \colon \Disk(r_0)^3 \to \Disk(r_1)$. 
				
				To go further we need some more estimates. The first one follows by standard estimates on ordinary differential equations.
				\begin{claim}For $(t, u, \zeta, \xi) \in \Disk(c)^4$,
					\begin{align*}
					\norm{\phi^{\zeta}(q) - \phi^{\zeta + \xi}(p)}
					&\lesssim \left(\abs{\xi} + \abs{t} + \abs{u}\right)\norm{p}.
					\end{align*}
					\label{claim:borne-flot}
				\end{claim}
				\begin{claim}If $c$ is small enough, for every $(t, u, \zeta, \xi) \in \Disk(c)^4$, \[\norm{\left(\deriv_\xi F\right)_{(t, u, \zeta, \xi)}^{-1}} \lesssim 1.\]\label{claim:inverse-implicite}
				\end{claim}
				\begin{proof}
					We have $\norm{\left(\deriv_\xi F\right)^{-1}} \lesssim \det\left(\deriv_\xi F\right)^{-1}\norm{\deriv_\xi F}^2$. Moreover $\det(\deriv_\xi F) = \abs{\partial_\xi F}^2 - \abs{\partial_{\overline{\xi}}F}^2$. For $(t, u, \zeta, \xi) \in \Disk(c)^4$,
					\begin{align*}\abs{\left(\partial_\xi F\right)(t, u, \zeta, \xi) - \left(\partial_\xi F\right)(0, 0, 0, 0)} &\lesssim \abs{t} + \abs{u} + \abs{\zeta} + \abs{\xi},\\
					\abs{\left(\partial_{\overline{\xi}} F\right)(t, u, \zeta, \xi)} &\lesssim \abs{t} + \abs{u} + \abs{\xi},\end{align*}
					where we use claim \ref{claim:borne-flot} in the second bound. So for $c$ small enough $\abs{\det\left(\deriv_\xi F_{(t, u, \zeta, \xi)}\right)} \gtrsim 1$ and $\norm{\deriv_\xi F(t, u, \zeta, \xi)} \lesssim 1$, which implies the result.
				\end{proof}

				For every $\zeta \in \Disk(r_0)$, $F(0, 0, \zeta, 0) = 0$, so $\xi(0, 0, \zeta) = 0$. Moreover, reducing $r_0$ and $r_1$ if necessary (as per the last part of lemma \ref{lem:implicit}), for every $t, u \in \Disk(r_0)^2$,
				\begin{align*}
				\abs{\frac{\partial \xi}{\partial t}(t, u, \zeta)} &= \abs{\left(\deriv_\xi F\right)^{-1}_{(t, u, \zeta, \xi)} \left(\partial_t F \right)(t, u, \zeta, \xi)} \lesssim 1,
				\end{align*}
				with $\xi = \xi(t, u, \zeta)$. Similar bounds hold for the differential with respect to $u$. Therefore $\abs{\xi(t, u, \zeta)} \lesssim \abs{t} + \abs{u}$ for every $(t, u, \zeta) \in \Disk(r_0)^3$.
				
				Using the bound proved above we can also refine claim \ref{claim:borne-flot}:
				\begin{claim}There exists a constant $\kappaLemProjLocale > 0$ such that if $r_0 > 0$ is small enough, depending only on the geometry of the foliation, for every $(t, u) \in \Disk(r_0)^2$ and $\zeta \in \Disk(r_0)$, 
					\[\norm{\phi^{\zeta}(q) - \phi^{\zeta + \xi}(p)} \leq e^\kappaLemProjLocale \norm{q - p},\]where $\xi = \xi(t, u, \zeta)$.
					\label{claim:derivee-proj}
				\end{claim}
				
				We now use the charts defined in the paragraph on the Poincaré-type metric. Fix $(t, u) \in \Disk(r_0)^2$. We can construct a map sending $\phi^\zeta(q)$ to $\phi^{\zeta+\xi}(p)$, which can be defined in such way that lemma \ref{lem:proj-locale} holds, adjusting constants appropriately.

				We now prove some other statements we will use. The first will allow us to compare heat kernels on close leaves. As was noted above, $\norm{F}_{\mathrm{C}^r} \lesssim_r 1$ for every non-negative integer $r$. This and claim \ref{claim:inverse-implicite} imply, by an easy induction, that for every non-negative integer $r$, the map $\xi$ and all its differentials up to order $r$ are bounded by constants depending only on the geometry of the foliation and $r$. This implies:
				\begin{claim}The map $\pi$ of lemma \ref{lem:proj-locale} can be constructed in such a way that in the charts of the paragraph on the Poincaré-type metric, in which it corresponds to $\zeta \mapsto \zeta+\xi$ using the notations above, we have $\norm{\pi - \mathrm{id}}_{\mathrm{C}^r} \lesssim_r \ell(p)^r\norm{q - p}$.
				\end{claim}
				
				Now we give statements related to the control of the holonomy, an essential input for the proof of the main theorem. For $p' \in \Disk(c)^2$ and $(t,u) \in \Disk^2(c)$, write $n_{p'}(V) = p' + tN(p') + u\norm{p'}\frac{\partial}{\partial z}$.
				Let $(t, u) \in \Disk^2(c)$. Given $p' \neq 0$ in the leaf $L_p$ with $\dist_g(p, p') \lesssim 1$, we can write $p' = \phi^\eta(p)$ for some $\eta \in \C$ with $\abs{\eta} \lesssim 1$. Then adjusting constants, there exists $\zeta$ with $\abs{\zeta} \lesssim 1$ such that $\eta = \zeta + \xi(t, u, \zeta)$. We can consider the map $h \colon (t, u) \mapsto h(t, u)$ such that $n_{p'}(h(t, u)) = \phi^{\eta - \xi}(n_p(t, u))$, which exists by the definition of $\xi$. Moreover
				\begin{align*}
				\norm{n_{p'}(h(t,u)) - p'} 
				&= \norm{\phi^{\eta - \xi}(n_p(t, u)) - \phi^\eta(p)}\\
				&\lesssim (\abs{t} + \abs{u})\norm{p}\\
				&\asymp \norm{p}\norm{n_p(t, u) - p}.
				\end{align*}
				But we also have $ \norm{n_{p'}(h(t,u)) - p'} \asymp \norm{p'}\norm{h(t,u)}$. Since $\norm{p'} \gtrsim \norm{p}$, we deduce that $\norm{h(t,u)} \lesssim \abs{t} + \abs{u}$. But the map $(t, u) \mapsto h(t, u)$ is the holonomy from $p$ to $p'$ in the parametrisation given by $n_p$ and $n_{p'}$. So we have the following important fact:
				\begin{lem}[Controlling the holonomy]
					There exists a small constant $\rho > 0$ and a large constant $ M > 0$ such that the following holds:
					
					Given a regular point $p \in P$ and $p' \in L_p$ with $\dist_g(p, p') \lesssim \ell(p)^{-1}$, the holonomy from $p$ to $p'$ is well defined on $S(p, \rho\dist(p,S))$  to $S(p', M\dist(p', S))$, and in the parametrisation given above, it gives a family of holomorphic diffeomorphisms on their image fixing $0$, which is well-defined on a ball of radius $\gtrsim 1$ and which is bounded in the uniform ($C^0$) norm, depending only on the geometry of the foliation, and not on $p$ and $p'$.
					\label{lem:controle-hol}
				\end{lem}
				Note that outside of a small neighbourhood of the singularities, the lemma is a straightforward consequence of compactness. We will also need the analogue of this lemma for longer paths:
				\begin{lem}
					Let $\gamma$ be a continuous path inside a leaf, starting at some point $p \in P \setminus S$, of length less than some $R > 0$. Let $\rho \in (0, \rhoSection)$  with 
					\[\exp\left(C\ell(p)e^{CR}\right)\rho \lesssim 1.\]
					Then the holonomy $h_\gamma$ along $\gamma$ is well-defined on $S(p, \rho)$ and for every $z_0, z_1 \in S(p, \rho)$, we have
					\[\frac{\dist\left(h_\gamma(z_0), h_\gamma(z_1)\right)}{\dist(p,S)} \leq C\exp\left(C\ell(p)e^{CR}\right)\frac{\dist(z_0, z_1)}{\dist(p,S)}.\]
					\label{lem:controle-hol2}
				\end{lem}
				The proof follows by induction from lemma \ref{lem:controle-hol}, along the lines of the proof of lemma \ref{lem:proj-globale} (indeed it is almost the same statement).
				
				Finally we also need the following lemma which allows to change transversals:
				\begin{lem}
					Let $p \in P \setminus S$ and $q \in M \setminus S$ with $\frac{\dist(p,q)}{\dist(p,S)} \lesssim 1$,
					and consider $h$ the holonomy at $q$ which sends a neighbourhood of $q$ in $S(p)$ to $S(q)$.
					Then if $z_0, z_1 \in S(p)$ with $\frac{\dist(z_0, q)}{\dist(q,S)}, \frac{\dist(z_1, q)}{\dist(q,S)} \lesssim 1$, we have
					\[\frac{\dist\left(h(z_0), h(z_1)\right)}{\dist(p,S)} \asymp \frac{\dist\left(z_0, z_1\right)}{\dist(q,S)}.\]
					\label{lem:changer-transversale}
				\end{lem}
				The proof follows from another simple application of the quantitative implicit function as above.
				\section{Negative Lyapunov exponents}
				\label{sect:negativite}
				In the rest of the article we consider $T$ a positive harmonic current directed by $\restr{\mathscr{F}}{P}$ given by theorem \ref{thm:fornaess-sibony}, and we denote by $\nu := T \wedge \mathrm{vol}_g$ the associated harmonic measure with the Poincaré-type metric $g$ on the leaves, normalising $T$ so that $\nu$ is a probability measure. 
				
				Let $\mathscr{M}$ be either $\mathscr{N} := N_{\restr{\mathscr{F}}{P}}$, the normal bundle of $\restr{\mathscr{F}}{P}$, or $\mathscr{L} := N_{P/M}$, the normal bundle to $P$ in $M$, both of which are holomorphic line bundles on $P \setminus S$, which can be extended to holomorphic line bundles on $P$ by Hartogs' phenomenon. On both these line bundles we can consider the hermitian metric induced by the ambient hermitian metric, which we denote by $\norm{\cdot}$. Because the foliation leaves the plane $P$ invariant, both these bundles are invariant by the holonomy. The aim of this section is to prove the following proposition:
				\begin{prop}Suppose that $M = \mathbb{P}^3(\C)$ and $P$ is a projective plane in $M$. Then  there exist $\lambda, \mu > 0$ such that for $\nu$-almost every $p \in P\setminus S$ and $W_p$-almost every path $\gamma \in \Gamma_p$,
					\begin{equation*}\lim_{t \longrightarrow \infty} \frac{1}{t}\log\norm{\restr{\deriv \left(h_\gamma^t\right)_{\gamma(0)}}{\mathscr{N}} \colon \mathscr{N}_{\gamma(0)} \to \mathscr{N}_{\gamma(t)} } = -\lambda < 0,
					\label{eq:lyapunovN}
					\end{equation*}
					\begin{equation*}\lim_{t \longrightarrow \infty} \frac{1}{t}\log\norm{\restr{\deriv \left(h_\gamma^t\right)_{\gamma(0)}}{\mathscr{L}} \colon \mathscr{L}_{\gamma(0)} \to \mathscr{L}_{\gamma(t)} } = -\mu < 0.
					\label{eq:lyapunovL}
					\end{equation*}
					\label{prop:lyapunov-negatif}
				\end{prop}
				We will follow the proof of Deroin-Dupon-Kleptsyn \cite{DDK}, although it is probably also possible to give an argument along the lines of Nguyên \cite{Nguyen}, replacing the explicit estimates near the singularities by implicit arguments along the lines of section \ref{sect:geometrie}. For our proof it will be useful to introduce two notions:
				\emph{Smooth metrics.} In addition to the induced metric, on $\mathscr{N}$, following Deroin-Dupont-Kleptsyn \cite{DDK}, we use also a "smooth metric": for every singular point $p_0$, there are holomorphic coordinates $(x,y,z)$ near $p_0$ such that $P$ is defined by $z = 0$, and the foliation is defined by a holomorphic vector field $X$ such that on $P$
				\[X = \alpha x\partial_x + \beta y \partial_y,\]
				where $\alpha, \beta \in \C$ are not colinear over $\R$. Then in these coordinates, we use the metric $\abs{\alpha x\deriv y - \beta y \deriv x}^2$ on $\mathscr{N}$. Note that the metric we use on $\mathscr{L}$ is given by $\abs{\deriv z}^2$ locally if the ambient hermitian metric is the usual metric in these coordinates, which we can assume, and this also gives a "smooth" metric. These metrics will be useful later on because when applied to local holomorphic sections of $\mathscr{N}$ or $\mathscr{L}$, they give harmonic functions near the singularities. We will denote them both by $\norm{\cdot}_s$ (so in the case of $\mathscr{L}$, $\norm{\cdot} = \norm{\cdot}_s$).
				
				\emph{Flat sections and the derivative of the holonomy. }For every $p \in P\setminus S$, and $s(p) \in \mathscr{M}_p$, we can consider for $p'$ close enough to $p$ in $P \setminus S$ the image $s(p')$ of $s(p)$ by the differential of the holonomy map between $p$ and $p'$, that is $\deriv\left(h_{pp'}\right)_p(s(p))$. The assignment $p' \mapsto s(p')$ defines a holomorphic section of $\mathscr{M}$ defined near $p$. We call such a section a \emph{flat section} near $p$ (it is a flat section for the Bott connection, a partial connection which can be used to differentiate in the direction of the foliation, but we will not need this). It is in fact uniquely determined by the value $s(p)$, wherever it is defined. Moreover, it is clear that a flat section can always be defined in a neighbourhood of a path in $P \setminus S$ inside a leaf.
				
				\begin{lem}Let $\gamma \colon [0, 1] \to L$ be a continuous path in a leaf $L$, and $\tilde{\gamma}$ its lift in a universal cover of a leaf $L\subset P$.  Then for all $t \in [0, 1]$,
					\[\log \norm{\restr{\deriv \left(h_\gamma^t\right)_{\gamma(0)}}{\mathscr{M}} \colon \mathscr{M}_{\gamma(0)} \to \mathscr{M}_{\gamma(t)}} \leq C\ell(\gamma(0))\exp\left(C\dist_g\left(\tilde{\gamma}(0), \tilde{\gamma}(t)\right)\right)\]
					for some constant $C > 0$ which depends only on the geometry of the foliation. Moreover, the same statement is true if we replace $\norm{\cdot}$ by $\norm{\cdot}_s$
					\label{lem:croissance-dhol}
				\end{lem}
				\begin{proof}
					Let $p_0 \in S$. Assume that in a neighbourhood $U$ of $p_0$ in $M$, there are holomorphic coordinates $(x, y, z)$ such that $P$ is locally given by $\{z = 0\}$. Assume that $\mathscr{F}$ is locally defined by a holomorphic vector field $X$ which, restricted to $P$, is given by $X(x, y, 0) = \alpha x \partial_x + \beta y\partial_y$ with $\alpha, \beta \in \C$ independent over $\R$.
					Therefore, in $M$ it can be written
					\[X(x, y, z) = \left(\alpha x + z\alpha'(x,y,z)\right) \partial_x + \left(\beta y + z\beta'(x, y, z)\right)\partial_y + z\gamma(x,y,z)\partial_z,\] for some holomorphic functions $\alpha', \beta', \gamma$. We assume that $\alpha', \beta', \gamma$ are bounded by some constant $C > 0$, as we may, slightly reducing the neighbourhood if necessary.
					
					Let $p \in P \cap U$ and consider the path $\gamma(t) := \phi^{t\zeta}(p)$ for some $\zeta \in \C$ with $\abs{\zeta} \lesssim 1$, where $(\phi^\zeta)$ is the flow of $X$. We consider $v \in \C^3$ with $v \in \left(N_{\restr{\mathscr{F}}{P}}\right)_p$ or $v \in \left(N_{P/M}\right)_p$ — in both cases, $\C v \perp X(p)$, and we want to estimate $\norm{\deriv \left(h_\gamma\right)_p(v)}_g$, for a certain choice of hermitian metric $g$.
					For every $v \in \left(\C X(p)\right)^{\perp}$ close enough to zero, there exists $\xi(v) \in \C$ with $\abs{\xi(v)} \lesssim 1$ such that \begin{equation}\left(\phi^{\xi(v)}(p + v) - \phi^{\zeta}(p)\right) \in \C N(\phi^{\zeta}(p)) + \C\partial_z.
					\label{eq:holonomie}\end{equation}
					In fact, by the implicit function theorem, as in section \ref{sect:geometrie}, the map $\xi$ is differentiable.
					
					\begin{align*}
					\deriv\left(h_\gamma\right)_p(v) 
					&= \restr{\frac{\deriv}{\deriv t}}{t = 0} \phi^{\xi(tv)}(p + tv)\\
					&= \left(\frac{\partial}{\partial\zeta}\phi^\zeta(p)\right) \xi'(0) + \left(\frac{\partial}{\partial p}\phi^\zeta(p)\right)\cdot v \\
					&=  \xi'(0)X\left(\phi^\zeta(p)\right) + \left(\frac{\partial}{\partial p}\phi^\zeta(p)\right)\cdot v.
					\end{align*}
					Let us now estimate $u(\zeta) := \frac{\partial}{\partial p}\phi^\zeta(p)$. We have $u'(\zeta) = \left(\frac{\partial}{\partial_p}X\right)\left(\phi^\zeta(p)\right)$, and for every point $q \in P$ we can identify $\left(\frac{\partial}{\partial_p}X\right)(q)$ with:
					\[\left( {\begin{array}{ccc}
						\alpha & 0 & \alpha'(q) \\
						0 & \beta & \beta'(q) \\
						0 & 0 & \gamma(q) \\
						\end{array} } \right).\]
					From this we deduce that $\norm{u} \lesssim e^{C\abs{\zeta}}$ for $\abs{\zeta} \lesssim 1$, using standard estimates on ordinary differential equations.
					
					Let us now compute $\xi'(0)$. We have
					\begin{align*}
					0
					&= \det\left(\restr{\frac{\deriv}{\deriv t}}{t = 0} \phi^{\xi(tv)}(p + tv),N(\phi^{\zeta}(p)),\partial_z \right)\\
					&= \det\left(\xi'(0)X\left(\phi^\zeta(p)\right) + \left(\frac{\partial}{\partial p}\phi^\zeta(p)\right)\cdot v,N(\phi^{\zeta}(p)),\partial_z \right)\\
					&= \xi'(0)\det\left(X\left(\phi^\zeta(p)\right),N(\phi^{\zeta}(p)),\partial_z \right) + \det\left(\left(\frac{\partial}{\partial p}\phi^\zeta(p)\right)\cdot v,N(\phi^{\zeta}(p)),\partial_z \right)\\
					&= \xi'(0)\norm{X\left(\phi^\zeta(p)\right)}^2 + \det\left(\left(\frac{\partial}{\partial p}\phi^\zeta(p)\right)\cdot v,N(\phi^{\zeta}(p)),\partial_z \right)
					\end{align*}
					Therefore
					\[\abs{\xi'(0)}\norm{X\left(\phi^\zeta(p)\right)} \lesssim e^{C\abs{\zeta}}, \]
					and so $\norm{\deriv\left(h_\gamma\right)_p(v)} \lesssim e^{C\abs{\zeta}}\norm{v}$.
					
					To prove the same statement for the metric $\norm{\cdot}_s$, note that on $\mathscr{N}$, at a point $p$, we have $\norm{v}_s \asymp \norm{p}\norm{v}$ for $v\in\mathscr{N}_p$, so the result follows from the computations above (in fact we do not even need bounds on $\abs{\xi'(0)}$). For $\mathscr{L}$ as already noted, $\norm{\cdot} = \norm{\cdot}_s$ (and here again we do not need bounds on $\abs{\xi'(0)}$).
				\end{proof}
				
				Let $p \in P \setminus S$ be a regular point, and let $\gamma \in \Gamma_p$ be a path, and define
				\[H_{t}^{\mathscr{M}}(\gamma) := \log \norm{\restr{\deriv \left(h_\gamma^t\right)_{\gamma(0)}}{\mathscr{M}} \colon \mathscr{M}_{\gamma(0)} \to \mathscr{M}_{\gamma(t)}},\]
				\[\tilde{H}_{t}^{\mathscr{M}}(\gamma) := \log \norm{\restr{\deriv \left(h_\gamma^t\right)_{\gamma(0)}}{\mathscr{M}} \colon \mathscr{M}_{\gamma(0)} \to \mathscr{M}_{\gamma(t)}}_s.\]
				
				Now we need the following lemma which is similar to a fact used by Nguyên \cite[section 4]{Nguyen}, and depends on Nguyên's theorem (theorem \ref{thm:nguyen}).
				\begin{lem}
					For any $A \geq 0$,  the function $F = F_A$ defined by 
					\[F(\gamma) := \sup_{0 \leq t \leq 1} \ell(\gamma(t))\exp\left(A\dist_g(\tilde{\gamma}(0), \tilde{\gamma}(t))\right),\]
					for any $\gamma \in \Gamma$ with lift $\tilde{\gamma}$ to the universal cover, is integrable with respect to the measure $\overline{\nu}$.
					\label{lem:integrabilite}
				\end{lem}
				\begin{proof}
					By lemma \ref{lem:ell-sing}, we have $F(\gamma) \leq \sup_{0\leq t \leq 1}\ell(\gamma(0))\exp\left(B\dist_g(\tilde{\gamma}(0), \tilde{\gamma}(t))\right)$ for some real number $B$. By Nguyên's theorem \ref{thm:nguyen}, the function $\gamma \mapsto \ell(\gamma(0))$ is in $\Leb^1(\overline{\nu})$. So it suffices to prove that for every $p \in M \setminus S$, the map $\gamma \mapsto \sup_{0\leq t \leq 1}\exp\left(B\dist_g(\tilde{\gamma}(0), \tilde{\gamma}(t))\right)$ is integrable with respect to $W_p$, with bounds uniform in $p$. For this we can use the bound of lemma \ref{lem:saut-brownien} and some simple computations (see Nguyên \cite[section 4]{Nguyen} for details).
				\end{proof}
				Therefore by lemma \ref{lem:croissance-dhol}, we obtain the following:
				\begin{lem}
					The functions
					\[\gamma \in \Gamma \mapsto \sup_{0 \leq t \leq 1} \abs{H_{t}^{\mathscr{M}}(\gamma)},\]
					\[\gamma \in \Gamma \mapsto \sup_{0 \leq t \leq 1} \abs{\tilde{H}_{t}^{\mathscr{M}}(\gamma)} \]
					are in $\Leb^1(\overline{\nu})$.
				\end{lem}
				
				\begin{proof}[Proof of proposition \ref{prop:lyapunov-negatif}]
					By Birkhoff's pointwise ergodic theorem, because $\overline{\nu}$ is ergodic, this implies that for $\overline{\nu}$-almost every $\gamma \in \Gamma$, the limit \[\lim_{t\to\infty} \frac{1}{t}\log\norm{\restr{\deriv \left(h_\gamma^t\right)_{\gamma(0)}}{\mathscr{M}} \colon \mathscr{M}_{\gamma(0)} \to \mathscr{M}_{\gamma(t)}}\] exists and is equal to $\int_{\Gamma} H_{t}^{\mathscr{M}}(\gamma)\deriv\overline{\nu}(\gamma)$, and similarly for $\norm{\cdot}_s$.
					
					Moreover, we have the following cocycle property: for every $s, t \geq 0$, and every $\gamma$:
					\[H_{s + t}^{\mathscr{M}}(\gamma) = H_{s}^{\mathscr{M}}(\gamma) + H_{t}^{\mathscr{M}}(\sigma_s(\gamma)),\]
					and similarly for $\tilde{H}_t^\gamma$. Therefore, integrating and using the shift-invariance of $\overline{\nu}$, we obtain that the map
					\[t\mapsto \int_{\Gamma} H_{t}^{\mathscr{M}}(\gamma)\deriv\overline{\nu}(\gamma)\]
					is linear in $t$; since it is also measurable, assume that it is given by $t \mapsto \lambda_{\mathscr{M}} t$. The map
					\[t\mapsto \int_{\Gamma} \tilde{H}_{t}^{\mathscr{M}}(\gamma)\deriv\overline{\nu}(\gamma)\]
					is also linear in $t$; assume that it is given by $t \mapsto \lambda_{\mathscr{M}}' t$. Then because $\overline{\nu}$-almost every path leaves a neighbourhood of the singularities for arbitrarily large times, on the complement of which $\norm{\cdot} \asymp \norm{\cdot}_s$, we have $\lambda_{\mathscr{M}}' = \lambda_{\mathscr{M}}$.
					
					We will recover $\lambda_{\mathscr{M}}$ by taking the limit as $t \to 0$. For this it is better to work with the metric $\norm{\cdot}_s$. 
					Take $p \in P \setminus S$, $\tilde{p}$ a lift of $p$ in $\tilde{L_p}$, and $\pi_p \colon \tilde{L_p} \to L_p$ the universal covering. Consider $\sigma$ the local flat section in a neighbourhood of $p$ described above, and consider the function $\theta$ which is defined locally by $\Delta_g\log\norm{\sigma}_s$ (here $\Delta_g$ is the Laplace-Beltrami operator for the Poincaré-type metric $g$). Note that $\theta$ does not depend on the choice of $\sigma$, and can actually be extended globally to a function on the leaf. Moreover, by the construction of the metric $\norm{\cdot}_s$, $\theta$ is zero near the singularities.
					
					For any path $\gamma \in \Gamma_p$, $H_{t}^{\mathscr{M}}(\gamma)$ only depends on the homotopy class with fixed endpoints of $\gamma$, so there exists a smooth function $\phi \colon \tilde{L_p} \to \R$ such that for any lift $\tilde{\gamma}$ of $\gamma$ starting at $\tilde{p}$, we have $H_{t}^{\mathscr{M}}(\gamma) = \phi(\tilde{\gamma}(t))$. In fact, consider a small neighbourhood of $\tilde{p}$, which can be identified with a small neighbourhood of $p$. Then through this identification, $\phi$ is given by $\log\norm{\sigma}_s$, where $\sigma$ is a flat section near $p$.
					
					Now we write:
					\begin{align*}\frac{1}{t}\int H_{t}^{\mathscr{M}}(\gamma)\deriv W_p(\gamma) 
					&= \frac{1}{t}\int H_{t}^{\mathscr{M}}(\pi_p \circ \tilde{\gamma})\deriv \tilde{W}_{\tilde{p}}(\gamma)\\
					&= \frac{1}{t}\int \phi(\tilde{\gamma}(t))\deriv \tilde{W}_{\tilde{p}}(\gamma)\\
					&= \frac{1}{t}\int \left(\int_0^t\Delta_g\phi(\tilde{\gamma}(t))\deriv s\right)\deriv\tilde{W}_{\tilde{p}}(\gamma)\\
					&= \int\frac{1}{t}\left(\int_0^t\theta(\gamma(s))\deriv s\right)\deriv W_{p}(\gamma),
					\end{align*}
					where we have used Dynkin's formula (see \cite{CandelGarnett} or \cite{CandelFoliationsII}) in the next to last equation.
					Let now $\eps > 0$. Because $\theta$ is smooth and zero near the singularities, there exists $\delta > 0$ such that for every $p, p'$ in the same leaf with $\dist_g(p, p') < \delta$, we have $\abs{\theta(p) - \theta(p')}<\eps$, and $\theta$ is bounded. Consider $P_t^\delta(p)$ the probability that a path starting from $p$ remains in $B_g(p,\delta)$ from times $0$ to $t$.
					Then 
					\begin{align*}
					\abs{\frac{1}{t}\left(\int_0^t\theta(\gamma(s))\deriv s\right) - \theta(p)} 
					&\leq P_t^\delta(p)\eps + 2(1 - P_t^\delta(p)) \norm{\theta}_\infty.
					\end{align*}
					Because the leaves have uniformly bounded geometry (see lemma \ref{lem:saut-brownien}), for $t$ small enough, we have $P_t^\delta(p) > 1 - \eps/\norm{\theta}_\infty$ for every $p$. Then
					\begin{align*}
					\abs{\frac{1}{t}\left(\int_0^t\theta(\gamma(s))\deriv s\right) - \theta(p)} 
					< 3\eps,
					\end{align*}
					and therefore, since $\eps > 0$ is arbitrary, we have uniform convergence over $p$ and thus
					\[\frac{1}{t}\int H_{t}^{\mathscr{M}}(\gamma)\deriv \overline{\nu}(\gamma)  \xrightarrow[t\to0]{} \int \theta \deriv\nu.\]
					Now we have, since the Chern class $c_1(\mathscr{M})$ can locally be represented by $-\frac{1}{\pi}\theta\mathrm{vol}_g$:
					\begin{align*}[T]\cdot c_1(\mathscr{M}) 
					&= -\frac{1}{\pi} \int \theta\deriv\nu \\
					&= -\frac{1}{\pi}\lambda_{\mathscr{M}}.
					\end{align*}
					For the foliation $\mathscr{F}$, we can define a degree $d > 1$ (see for instance Loray-Rebelo \cite{LorayRebelo}). We have $\mathscr{L} = N_{P/M} \simeq \mathscr{O}_P(1)$, while $\mathscr{N} = N_{\restr{\mathscr{F}}{P}} \simeq \mathscr{O}_P(d + 2)$, where $d > 1$ is the degree of $\restr{\mathscr{F}}{P}$. Therefore:
					\begin{align*}\lambda_{\mathscr{N}} 
					&= -\pi[T]\cdot c_1(\mathscr{O}_P(d + 2))\\
					&= -(d+2)\pi[T]\cdot c_1(\mathscr{O}_P(1))\\
					&= (d+2) \lambda_{\mathscr{L}}.
					\end{align*}
					All the computations we have just made can also be done with the Brownian motion on the leaves defined using the Poincaré metric instead of the Poincaré-type metric. Let $\nu' = T \wedge \mathrm{vol}_{g_P}$. The measure $\theta\deriv \nu$ is independent of the metric, Poincaré or Poincaré-type $g$. So if we denote the Lyapunov exponent $\lambda^P_\mathscr{M}$, we have, adding a factor taking into account the fact that we do not necessarily have $\nu'(M) = 1$, when we apply the ergodic theorem:
					\begin{align*}
					\lambda^P_\mathscr{M} 
					&= \frac{1}{t\nu'(M)}\int H_{t}^{\mathscr{M}}(\gamma)\deriv \overline{\nu'}(\gamma)\\
					&\xrightarrow[t\to0]{} \frac{1}{\nu'(M)}\int \theta \deriv\nu\\
					&= \frac{1}{\nu'(M)}\lambda_{\mathscr{M}}.
					\end{align*}
					By the work of Nguyên \cite{Nguyen} or Deroin-Dupont-Kleptsyn \cite{DDK}, we can compute $\lambda_\mathscr{N}$:
					\[\lambda_\mathscr{N}^P = -\frac{d+2}{d-1}.\]
					Therefore we have:
					\[\lambda_{\mathscr{N}} = -\nu'(M)\frac{d+2}{d-1},\]
					\[\lambda_{\mathscr{L}} = -\nu'(M)\frac{1}{d-1}.\]
				\end{proof}

				\section{Contraction}
				\label{sect:contraction}
				In this section we prove the following proposition \ref{prop:contraction}, which is the basic dynamical ingredient for the rest of the proof. The proof is very similar to that of Deroin-Kleptsyn \cite[paragraph 2.1]{deroinkleptsyn}. The idea is to discretise the holonomy by cutting the Brownian path in short pieces, and applying a simple contraction lemma  (lemma \ref{lem:contraction}) — a rather trivial instance of Pesin theory \cite{BarreiraPesin} —, which proves, roughly stated, that when composing maps from a given family with a common fixed point of negative upper-Lyapunov exponent, the fixed point is attracting.
				
				Compared to the work of Deroin-Kleptsyn, the proof is slightly more complicated, because we have to take care of the (mild) non-conformality, and more importantly, to deal with the singularities. The idea is to sub-discretise further, with a step comparable to $\abs{\log \dist(\cdot, S)}^{-1}$, which therefore becomes smaller as one approaches a singularity, so that the holonomy remains controlled (see lemma \ref{lem:controle-hol}) — indeed, the contraction lemma with iteration of different maps requires a compactness-type assumption on the family of maps. The key idea (originally from Deroin-Dupont-Kleptsyn \cite{DDK}) is that up to a finite number of steps, the number of maps composed will be approximately a Birkhoff sum for $\abs{\log \dist(\cdot, S)}^{-1}$, which is integrable by Nguyên's theorem (theorem \ref{thm:nguyen}), and therefore almost surely grows linearly with time. This implies that the subdiscretisation will not prevent contraction.
				
				\begin{prop}
					Assume that we are in the context of proposition \ref{prop:lyapunov-negatif}.
					Let $\alpha = \min\left\{\lambda, \mu\right\}$. Then for every $\eps \in (0, \alpha)$, for $\nu$-almost every $p \in P\setminus S$, and $W_p$-almost every path $\gamma \in \Gamma_p$, there exists a constant $C > 0$, a radius $\rho > 0$ such that for every $t \geq s \geq 0$, the map $h_\gamma^{s,t}$ is defined on $S(\gamma(0), \rho e^{-\eps s})$ and for every $z_0, z_1 \in S(0, \rho e^{-\eps s})$, we have
					\[\frac{\dist(h_\gamma^{s,t}(z_0),h_\gamma^{s,t}(z_1))}{\dist(\gamma(t), S)} \leq Ce^{-(\alpha-\eps)(t-s)}e^{\eps s} \frac{\dist(z_0,z_1)}{\dist(\gamma(s), S)}.\]
					\label{prop:contraction}
				\end{prop}
				We will use the following simple cocycle lemma, a variant of a fact sometimes used to prove the Oseledets ergodic theorem (see for instance Barreira-Pesin \cite[theorem 1.3.12]{BarreiraPesin}).
				\begin{lem}
					Let $(A_n)_{n \geq 0}$ be a sequence of upper-triangular matrices in $\mathrm{GL}_2(\C)$, and write $A_n = \left( {\begin{array}{cc}
						a_n & c_n \\
						0 & b_n \\
						\end{array} } \right)$ for some sequences of complex number $(a_n), (b_n), (c_n)$. Assume that there exist $M > 0$ and $\lambda, \mu > 0$ such that
					\[\abs{c_n} \leq M, \]
					\[\lim \frac{1}{n} \sum_{i = 0}^{n- 1} \log\abs{a_i} = -\lambda < 0,\]
					\[\lim \frac{1}{n} \sum_{i = 0}^{n- 1} \log\abs{b_i} = -\mu < 0.\]
					Let $\alpha = \min\{\lambda, \mu\}$. Then for every $\eps > 0$, there exists $C > 0$ such that for every integers $m \geq 0$ and $n \geq 0$,
					\[\norm{A_{n + m}A_{n + m - 1}\cdots A_{m + 1}} \leq C e^{\eps m} e^{-(\alpha - \eps) n}.\]
					\label{lem:lyapunov}
				\end{lem}
				\begin{proof}
					Let $\eps > 0$. There exists $C \geq 1$ such that for every integer $n \geq 0$, 
					\[C^{-1}e^{-\lambda n - \eps n} \leq \abs{a_{n-1}\cdots a_{0}} \leq Ce^{-\lambda n+ \eps n},\]
					\[C^{-1}e^{-\mu n - \eps n} \leq \abs{b_{n-1}\cdots b_{0}} \leq Ce^{-\mu n+ \eps n}.\]
					Therefore, for some constant $D \geq 1$, we have, for every integers $m, n \geq 0$,
					\[\abs{a_{n+m}\cdots a_{m + 1}} \leq De^{2\eps m}e^{-(\lambda - \eps) n},\]
					\[\abs{b_{n + m}\cdots b_{m + 1}} \leq De^{2\eps m}e^{-(\mu - \eps) n}.\]
					Now, for integers $n, m \geq 0$, let us bound the the upper-right coefficient $C_{n,m}$ of $A_{n + m}A_{n + m - 1}\cdots A_{m + 1}$. We have, with obvious conventions for the degenerate indices:
					\begin{align*}
					\abs{C_{n,m}} &= \abs{\sum_{j = m}^{n + m - 1} a_{n + m}\cdots a_{j + 1} c_{j} b_{j - 1}\cdots b_{m + 1}} \\
					&\leq \sum_{j = m}^{n + m - 1} D^2M e^{2\eps j}e^{2\eps m}e^{-(\lambda - \eps)(n + m - j)}e^{-(\mu-\eps) (j - 1 - m)}\\
					&\lesssim e^{2\eps m}e^{-(\alpha - 2\eps) n}
					\end{align*}
					where the implicit constant can depend on $D, M, \eps, \lambda, \mu$ but not on $m$ and $n$. Putting all the estimates together gives the claim.
				\end{proof}
				\begin{lem}[Contraction lemma]
					Choose some norm $\norm{\cdot}$ on $\C^2$.
					Let $M > 0$ and $r > 0$, $\alpha > 0$. Assume that we have a sequence $(f_n)_{n\in \N}$ of $\mathrm{C}^2$ maps from $B_{\norm{\cdot}}(0, r)$ to $\C^2$ which are diffeomorphisms onto their image and fix $0$, and that for every $n \in \N$ and $x \in B(0, r)$, we have
					$\norm{\deriv^2(f_n)_x} \leq M$.
					
					For every integer $n \geq 0$, let $A_n = \deriv \left(f_n\right)_0$. Assume that for every $\eps > 0$, there exists $C > 0$ such that for every integers $m \geq 0$ and $n \geq 0$,
					\[\norm{A_{n + m}A_{n + m - 1}\cdots A_{m + 1}} \leq C e^{\eps m} e^{-(\alpha - \eps) n}.\]
					
					Then for every $\eps \in (0, \alpha)$, there exists $\rho \in (0, r)$ and $C>0$, depending only on $M$, $r$, $\alpha$ and $\eps$, such that for any integers $m \geq 0$, and $n \geq 0$,  $f_{m + n - 1} \circ \ldots\circ f_m$ is well defined and $Ce^{-(\alpha-\eps)n}e^{\eps m}$-Lipschitz on $B(0, \rho e^{-\eps m})$.
					\label{lem:contraction}
				\end{lem}
				
				\begin{proof}
					This is a very simple instance of Pesin theory, see e.g. Barreira-Pesin \cite[chapter 7]{BarreiraPesin}. Let $\eps \in (0, \alpha/4)$ with $\eps < 1$ and apply the hypothesis: there exists $C > 0$ such that for every integer $m > 0$ and $n \geq 0$,
					\begin{equation}\norm{A_{n + m}A_{n + m - 1}\cdots A_{m + 1}} \leq C e^{\eps m} e^{-(\alpha - \eps) n}.\end{equation}
					Now for any $m \in \N$, consider the (Lyapunov) norm $\norm{\cdot}_m^{'}$ defined by
					\[\norm{u}_m^{'} := \sum_{k = 0}^\infty \norm{A_{m + k - 1}\cdots A_m u}e^{(\alpha - 2\eps)k},\]
					for every $u \in \C^2$, which is well-defined by the previous equation, and moreover for every $u \in \C^2$, $\norm{u} \leq \norm{u}_m^{'} \leq C_\eps e^{\eps m}\norm{u}$. The important property of this norm is that for every $m \in \N$, $u \in \C^2$, we have $\norm{A_mu}^{'}_{m + 1} \leq e^{-(\alpha - 2\eps)}\norm{u}^{'}_m$.
					
					For $n \in \N$, write $f_n = A_n + R_n$, where $\deriv_0 (R_n) = 0$. Then for every $u \in B(0, r)$, we have $\norm{R_n(u)}^{'}_{n + 1} \leq C_\eps Me^{\eps (n + 1)}\left(\norm{u}^{'}_{n}\right)^2$, and therefore
					\[\norm{f_n(u)}^{'}_{n + 1} \leq e^{-(\alpha - 2\eps)}\norm{u}^{'}_n + C_\eps Me^{\eps (n + 1)}\left(\norm{u}^{'}_{n}\right)^2.\]
					Now choose $\delta > 0$ such that $e^{-(\alpha - 3\eps)} + \delta e^{2\eps}\leq 1$ and $e^{-(\alpha-2\eps)} + \delta \leq e^{-(\alpha-3\eps)}$. Let $x \in B(0, r)$ with $\norm{x}'_m \leq \delta C_\eps^{-1}M^{-1}e^{-\eps m}$. By induction, for every $n \geq 0$, we have
					\[\norm{f_{m + n - 1} \circ \ldots \circ f_m(x)}'_{n + m} \leq \delta C_\eps^{-1}M^{-1}e^{-\eps (m + n)}.\]
					By another induction, we deduce that
					\[\norm{f_{m + n - 1} \circ \ldots \circ f_m(x)}'_{n + m} \leq e^{-(\alpha-3\eps)n}\norm{x}'_m.\]
					Therefore:
					\[\norm{f_{m + n - 1} \circ \ldots \circ f_m(x)} \leq C_\eps e^{-(\alpha-3\eps)n}e^{\eps m}\norm{x}\]
					for any $x \in B(0,r)$ such that $\norm{x} \leq \delta C_\eps^{-1}M^{-1}e^{-\eps m}$.
					
					Now for such an $x$ we can also write for any $n \geq 0$, \[x_n := f_{m + n - 1} \circ \ldots \circ f_m(x),\]
					\[L'_n := \deriv \left(f_{m + n - 1}\right)_{x_{n - 1}} \circ \ldots \circ \deriv \left(f_m\right)_{x_0},\]
					\[L_n := A_{m+n - 1}\cdots A_m,\]
					\begin{align*}&\norm{L'_n - L_n}'_{n + m}\\
					&\leq \norm{\deriv \left(f_{m + n - 1}\right)_{x_{n-1}} - \deriv \left(f_{m + n - 1}\right)_0}'_{m+n}\norm{L'_{n-1}}'_{m+n} + \norm{A_{m + n - 1}\left(L'_{n-1}-L_{n-1}\right)}'_{m+n}\\
					&\leq C_\eps Me^{\eps(m+n)}\norm{x_{n-1}}\left(\norm{L_{n-1}}'_{m+n} + \norm{L'_{n-1} -L_{n-1}}'_{m+n}\right) + e^{-(\alpha-2\eps)}\norm{L'_{n-1}-L_{n-1}}'_{m+n}.
					\end{align*}
					Further reducing $\delta$ if necessary, we deduce by induction that for every $n \geq 0$,
					$\norm{L'_n-L_n}'_{m+n} \leq e^{-(\alpha-3\eps)n}e^{\eps m}\norm{x}$, and therefore $\norm{L'_n} \leq C_\eps e^{-(\alpha-3\eps)n}e^{\eps m}$. Therefore $f_{m + n - 1} \circ \ldots \circ f_m$ is $C_\eps e^{-(\alpha-3\eps)n}e^{\eps m}$-Lipschitz on a ball of the form $B(0,\rho e^{-\eps m})$ for some $\rho > 0$.
				\end{proof}

				\begin{proof}[Proof of proposition \ref{prop:contraction}]
					Choose some $\delta \in (0,1)$.  Let $\gamma \in \Gamma$ be a path satisfying the (generic) negative Lyapunov exponent assumption of the proposition. Consider, for any $n \in \N$, the point $p_n = \tilde{\gamma}(n\delta)$, where $\tilde{\gamma}$ is a lift of $\gamma$ to the universal covering space $\tilde{L}$ of the leaf of $p := \gamma(0)$. For every two points $a, b \in \tilde{L}$, recall that we write $h_{a,b}$ for the holonomy relative to the geodesic from $a$ to $b$. Then we can write, for any integer $n \geq 1$:
					\[h_{p_0p_{n - 1}} = h_{p_{n - 2}p_{n-1}}\circ \ldots \circ h_{p_0p_1}.\]
					For any integer $i$ with $0 \leq i \leq n-1$, we can also write
					\[h_{p_ip_{i + 1}} = h_{a_{K - 2}a_{K-1}}\circ \ldots \circ h_{a_0a_1}\]
					for some points $a_j = a_j^i$, $j = 0, 1, \ldots, K-1$ with $\dist_g(a_j, a_{j + 1}) \lesssim \ell(p_i)^{-1}$  where $K_i$ is the upper integer part of $C\dist_g(\tilde{\gamma}(i\delta),\tilde{\gamma}((i + 1)\delta))\ell(p_i)$. Here the constant $C > 0$ depends only on the geometry of the foliation and is chosen so that we may apply lemma \ref{lem:controle-hol} to the $h_{a_ja_{j + 1}}$, for $i \geq 1$ and $ 0 \leq j \leq K_i$. Note that it is uniformly bounded and defined on balls of radius uniformly bounded away from zero, depending only on the geometry of the foliations, so we may apply the Cauchy inequalities to get $\mathrm{C}^2$ bounds.
					
					Concatenating all these finite sequences, we obtain a sequence $(f_n)_{n \in \N}$ of maps defined on $B(0,r)$ for some $r \gtrsim 1$ and bounded in $\mathrm{C}^2$ norm by a constant $M \lesssim 1$ (using the parametrisation of lemma \ref{lem:controle-hol}). For any $n \in \N$, the map $f_{n - 1} \circ \ldots \circ f_0$ is the holonomy along $\gamma$ from times $0$ to $t(n)$ for some $t(n) \geq 0$. The differentials $\deriv_0 f_n$ can be seen as upper-triangular matrices whose diagonal entries define two sequences $(a_n)$ and $(b_n)$ and by assumption we have:
					\[\lim \frac{1}{t(n)} \log\left(|a_{n-1}\ldots a_0|\right) < 0, \]
					\[\lim \frac{1}{t(n)} \log\left(|b_{n-1}\ldots b_0|\right) < 0. \]
					Now we show that $t(n)/n$ has a non-zero limit. The $K_i$ defined above is a certain non-negative function (independent of $i$) of $\sigma_{i\delta}(\gamma)$ which is certainly smaller than the function of lemma \ref{lem:integrabilite} and therefore integrable with respect to $\overline{\nu}$.
					Therefore, by the Birkhoff ergodic theorem, for $\overline{\nu}$-almost any $\gamma$, $\left(K_0 + \ldots + K_{p - 1}\right)/n$ converges as $p\to\infty$ to a limit $C > 0$ (independent of $\gamma$). Note moreover that for every integer $p \geq 1$, we have $t(K_0 + \ldots+K_{p-1}) = p\delta$ and the function $n \mapsto t(n)$ is increasing, so the ratio $t(n)/n$ converges as $n \to \infty$ to $C\delta$. Therefore we have
					\[\lim \frac{1}{n} \log\left(|a_{n-1}\ldots a_0|\right) < 0, \]
					\[\lim \frac{1}{n} \log\left(|b_{n-1}\ldots b_0|\right) < 0, \]
					so we may apply lemma \ref{lem:lyapunov}.
					
					Let now $t \geq s \geq 0$. We choose $m$ the smallest integer such that $m\delta \geq s$ and $n$ the largest integer such that $n\delta \leq t$. We may write the holonomy $h$ along $\gamma$ between times $s$ and $t$ in the parametrisation of lemma \ref{lem:controle-hol} as $h = a \circ h' \circ b$, where $a$ and $b$ are holonomies on short intervals of time such that lemma \ref{lem:controle-hol} applies, and the map $h'$ is the composition $f_{\ell} \circ \cdots f_k$ for some integers $k, \ell$ with $0 \leq k \leq \ell$. Write for any $n \in \N$, $S_n = K_0 + \ldots + K_{n - 1}$. By definition, we have
					\[S_{m - 1} < k \leq S_m \leq S_n \leq \ell < S_{n + 1}.\]
					There exists an integer $n_0$ such that for every integer $n \geq n_0$, we have $Cn - \eps n \leq S_n \leq Cn + \eps n$. Therefore if $k > n_0$, which we can assume, adjusting the constants as necessary, we have $(C - \eps)m - C + \eps \leq k \leq (C + \eps)m$ and $\ell - k \geq S_{n} - S_{m - 1} \geq C(n -m) - \eps(n + m) + C + \eps$. 
					By lemma \ref{lem:contraction}, there exists $\rho > 0$, depending only on the geometry of the foliation and $\eps$, such that for any $x \in B(0, \rho e^{-\eps k})$, $h'(x)$ is well defined, belongs to $B(0, \rho e^{-\eps k})$ and for every $\theta \in [0, \rho e^{-\eps k})$, the map $h'$ sends the ball $B(0, \theta)$ inside $B(0, Ce^{-(\alpha - \eps)(\ell - k + 1)}e^{\eps k}\theta)$. Using the estimates on $k$ and $\ell - k$, this gives the result for $h'$, adjusting $C$ and $\rho$. The result for $h$ follows by further adjusting the constants to account for $a$ and $b$. The Lipschitz part of the proof follows by similar arguments.
				\end{proof}
				
				\section{Similarity and conclusion of the proof}
				\label{sect:similarite}
				In this section, we prove the similarity. This is quite technical. We follow the proof of Deroin-Kleptsyn \cite{deroinkleptsyn}, making essential use of discretisation of Brownian path. This is necessary because we lose genericity if we consider the small movements of Brownian motion. But in our singular context, the discretisation is more delicate, just like the contraction part was: it is necessary to have a subdiscretisation, which means first discretising with a time step $\delta$, and then further, with a step smaller than $\delta\ell(p)^{-1}$ for a segment of Brownian path which starts at $p$ (contrary to the contraction part of the proof, the subdiscretisation is done with a step which is \emph{smaller}, and not comparable to, $\delta\ell(p)^{-1}$, for technical reasons). Therefore we suggest reading the simpler case treated by Deroin-Kleptsyn \cite[proof lemma 2.4]{deroinkleptsyn}, which is already quite involved. We will refer to it several times.
				
				There are several steps to the proof, which can be described as follows:
				\begin{enumerate}
					\item First, we prove that for two close points $p$, $q$ with $q \in S(p)$, there are sets on which the heat kernels are close (proposition \ref{lem:comparaison}).
					\item For the second step, we consider sets of "good" paths, which satisfy the contraction property of proposition \ref{prop:contraction}, with exponent $\beta > 0$ and such that at time $t$, the point $\gamma(t)$ is at a distance bounded from below from $S$ by roughly $e^{-\kappa\beta t}$, for some small parameter $\kappa > 0$. By Nguyên's theorem \ref{thm:nguyen} and proposition \ref{prop:contraction}, for generic points in $P \setminus S$, the probability that a path is good is close to $1$. Choose such a point $p$.
					\item Then we prove (lemma \ref{lem:queue-A}) that if we take $p'$ \emph{any} (not just almost any) point close to $p$ in $L_p$, the probability that a path starting from $p'$ is good is still close to $1$.
					\item Now we try to compare paths on different leaves. For this, it is convenient to discretise paths (definition \ref{def:A}): for a path $\gamma$, we consider only the points $\gamma(t)$ for $t \in T$, where $T$ is a certain discrete set. The set $T$ is such that we consider shorter intervals of times as time goes to infinity, depending on $\kappa$, to account for the fact that the path might be at distance of order $e^{-\kappa\beta t}$ to the singularity at time $t$. The aim of this is to keep control of the heat kernel in proposition \ref{lem:comparaison} as we might get very close to the singularity, where the behaviour of leaves can be wild. 
					
					Moreover, taking \emph{any} (not just almost any) point $q$ in $S(p)$ close enough to $p$, we ask that the $\gamma(t)$, for $t \in T$, belong to the sets where the heat kernel on $L_p$ and $L_q$ are close, as in step 1. Lemma \ref{lem:proba-A} shows that we can impose this without losing much measure on the set of paths.
					\item Then, we consider the map which takes a (discretised) path starting at $p$ and projects it to a path starting at $q$, and we show that this map is close to being "absolutely continuous", assuming $p$ and $q$ are very close. This is the content of lemma \ref{lem:absolue-continuite}. To formulate this statement, it is convenient to have defined discretised paths starting from $q$ (definition \ref{def:B}).
					
					From this, we deduce that if $q$ and $p$ are close, if a property depending only on times in $T$ is true for generic paths starting from $p$, then it is true generically for paths starting at $q$ (corollary \ref{coro:genericite}).
					\item Using all this, we can transfer generic properties from the point $p$ to close points $q$. Generically, we can assume that paths starting from $p \in P \setminus S$ are good and are distributed according to $\nu$, by the ergodic theorem. Then we deduce that the same is true for \emph{every} (crucially, not just \emph{almost} every) close enough points $q$ (corollary \ref{lem:similarite2}) in $M\setminus S$. To prove this, it is important to be able to consider various arbitrarily small steps of discretisation (depending on the modulus of continuity of functions against which we will test our measures), which we can do by lemma \ref{lem:red-delta} and lemma \ref{lem:dediscretisation}. The proofs of these lemma must involve recovering information on the path at all times from information at the discretised times, which might worsen the contraction properties and the distance from the singularities, but in a controlled way, depending on $\kappa$.
				\end{enumerate}
				Implementing all these steps, we have an open set in $M\setminus S$ where the behaviour of paths satisfies exponential contraction and distribution according to $\nu$. 
				
				In this section we assume that \ref{prop:lyapunov-negatif} holds, that is, there are two negative Lyapunov exponents $-\lambda$ and $-\mu$. We fix in all this section $\alpha > 0$ with $\alpha < \min\{\lambda,\mu\}$.
				\subsection{Comparing heat kernels on close leaves.}
				
				\begin{prop}
					There exist constants $\CstALemComparaison, \CstBLemComparaison > 0$ and $\cstLemComparaison > 0$, such that the following holds:
					
					Let $\delta \in (0, 2)$. Let $p \in P \setminus S$ and $q \in M \setminus S$ with $q \in S(x)$. Let $\tilde{p}, \tilde{q}$ be lifts of $p$ and $q$ in the universal covers of their leaves. For every $z$ in the universal covering of a leaf, consider the measure $\nu_z^\delta = p_g(z, \cdot ; \delta)\mathrm{d}\mathrm{vol}_g$ on $\tilde{L}_z$, where we recall that $p_g$ is the heat kernel for the metric $g$.
					Let $R > \RLemSautBrownien$ be such that \begin{equation}\exp\left(\CstALemComparaison\ell(s)e^{\CstALemComparaison R}\right)\dist(p, q) \leq \epsLemComparaison.
					\label{eq:comparaison}\end{equation} Then there exists $E \subset B_g(\tilde{p}, R) \subset \tilde{L}_p$ such that we have:
					\begin{itemize}
						\item $\nu_{\tilde{p}}^\delta(E) \geq 1 - \eps$ ;
						\item  the measure $\restr{\nu_{\tilde{p}}^\delta}{B_g(\tilde{p}, R)}$ is absolutely continuous with respect to the measure $\left(\phi_{\tilde{p}\tilde{q}}\right)_{*}\restr{\nu_{\tilde{q}}^\delta}{B_g(\tilde{q}, 2R)}$, and the Radon-Nikodym derivative satisfies \[\abs{\frac{\mathrm{d}\restr{\nu_{\tilde{p}}^\delta}{B_g(\tilde{p}, R)}}{\mathrm{d}\left(\left(\Phi_{\tilde{p}\tilde{q}}\right)_{*}\restr{\nu_{\tilde{q}}^\delta}{B_g(\tilde{q}, 2R)}\right)} - 1} \leq \eps\] at every point in $E$,
					\end{itemize}
					where
					\[\eps = \CstBLemComparaison\delta^{-\CstBLemComparaison}\left(\exp\left(\CstBLemComparaison\ell(p)e^{\CstBLemComparaison R}\right)\dist(p, q)  + \exp\left(-\cstLemComparaison\frac{R^2}{\delta}\right)\right)^{1/2}.\]
					
					\label{lem:comparaison}
				\end{prop}
				Below we will write the set in question $E(p, q, R)$ if the condition of the proposition are satisfied. 
				\begin{proof}
					Considering lemma \ref{lem:proj-globale}, we can define projections. We need some information about their construction, which is only local, and is contained in claim \ref{claim:derivee-proj} with $r = 5$. Choosing $C > 0$ in the statement large enough, we can assume $\norm{\Phi_{pq}}_{C^5} \lesssim 1$ in the norm of claim \ref{claim:derivee-proj}. Consider the metric $g'$, which coincides with $\left(\Phi_{pq}\right)_{*}g$ in the ball of center $p$ and radius $R$, and with $g$ outside of a larger ball of center $p$ and radius $2R$. This metric can be obtained by convex combination with coefficients defined by cut-off functions. Then $g$ and $g'$ are in a compact convex set for the $\mathrm{C}^4$ topology, depending only on the geometry of the foliation.
					
					Consider then the associated heat kernels $p_g$ and $p_{g'}$. Because the construction is explicit (see Candel-Conlon \cite[appendix B.6]{CandelFoliationsII} for an overview, or Berger-Gauduchont-Mazet \cite{BergerGauduchontMazet} for details), it depends differentiably on the metric when the time is bounded, with a polynomial singularity in time. So for every $\delta \in (0, 1)$, we have $\abs{p_g(\cdot, \cdot ; \delta) - p_{g'}(\cdot, \cdot;\delta)} \leq C\delta^{-C}\norm{g - g'}_{C^4}$. But we have \[\norm{\Phi_{pq} - \mathrm{id}}_{C^5} \lesssim \exp\left(C\ell(p)e^{CR}\right)\dist(p, q) \lesssim 1,\] which implies that $\norm{g - g'}_{C^4} \lesssim\exp\left(C\ell(s)e^{CR}\right)\dist(p, q)$, and so \[\abs{p_g(p, \cdot ; \delta) - p_{g'}(p, \cdot;\delta)} \lesssim \delta^{-C}\exp\left(C\ell(p)e^{CR}\right)\dist(p, q)\] for every $\delta \in (0, 1)$.
					The rest of the proof is just like in the work of Deroin-Kleptsyn \cite[proof of proposition 5.1]{deroinkleptsyn}.
				\end{proof}

				\subsection{Similarity of Brownian motion}

				In this section, we prove technical results on sets of paths to implement the "similarity" part of Deroin-Kleptsyn's method. Note that we will work with paths in universal coverings of leaves, but this will be slightly awkward notationally, and the reader may ignore this difficulty by assuming that all the leaves are simply connected. We will abuse notations and often write distances between points or sets, one of which is in a universal covering of a leaf, e.g. $\dist(p', S)$ for some $p' \in \tilde{L}_p$. In this case it should be understood that this means $\dist(\pi_p(p'),S)$, where $\pi_p \colon \tilde{L}_p \to L_p$ is a fixed universal covering map. Similarly, we can define $\ell(p')$, etc.
				
				\begin{defn}
					Let $p \in M \setminus S$ and $\tilde{p}$ be a lift in $\tilde{L}_p$. Let $\beta, \kappa > 0$ and $L > 0, B > 1$ and $\rho \in (0, \rhoSection)$. Define $\mathscr{A}_{\tilde{p}}(\kappa, B ; \beta, \rho, L)$ to be the set of paths $\gamma \in \tilde{\Gamma}_{\tilde{p}}$ such that for every $t \geq s \geq 0$:
					\begin{enumerate}
						\item the holonomy $h_{\gamma}^{s,t}$ along $\gamma$ between times $s$ and $t$ maps $S(\gamma(s),\rho e^{-\kappa\beta s})$ to $S\left(\gamma(t),L\rho e^{-\beta (t - s)}\right)$, and moreover, for any $z_0, z_1 \in S(\gamma(s),\rho e^{-\kappa\beta s})$ we have
						\[\frac{\dist\left(h_{\gamma}^{s,t}(z_0), h_{\gamma}^{s,t}(z_1)\right)}{\dist(\gamma(t),S)} \leq Le^{-\beta(t-s)}e^{\kappa\beta s}\frac{\dist(z_0, z_1)}{\dist(\gamma(s),S)}, \]
						where we have abused notation and written $h_\gamma^{s,t}$ for the holonomy of the projection of $\gamma$ on $L_p$,
						\item $\ell(\gamma(t)) \leq \kappa \beta t + B$.
					\end{enumerate}
					
					We also define $\mathscr{A}_{\tilde{p}}(\kappa; \beta)$ to be the union of the $\mathscr{A}_{\tilde{p}}(\kappa, B ; \beta, \rho, L)$ with $B, \rho, L$ as above. That is, $\mathscr{A}_{p}(\kappa; \beta)$ is the set of paths with exponential contraction of transverse sections with exponent $\beta$, and which do not approach the singularity too fast exponentially, depending on $\kappa$.
					\label{def:Acontinu}
				\end{defn}

				\begin{lem}
					Let $p \in M \setminus S$ and $\tilde{p}$ be a lift in $\tilde{L}_p$. Let $\kappa, \beta \asymp 1$ and $L > 0, B > 1$ and $\rho \in(0,\rhoSection)$. Let $\eta > 0$. Then if $B'$ and $L'$ are large enough and $\rho' > 0$ is small enough depending on $B$, for $p' \in \tilde{L}_p$ close enough to $\tilde{p}$, we have $\tilde{W}_{p'}(\mathscr{A}_{p'}(\kappa, B' ; \beta, \rho', L')) \geq \tilde{W}_{\tilde{p}}(\mathscr{A}_{\tilde{p}}(\kappa, B ; \beta, \rho, L)) - \eta$.
					\label{lem:queue-A}
				\end{lem}
				\begin{proof}
					The proof is an adaptation of an argument of Deroin-Kleptsyn \cite[paragraph 5.2.1]{deroinkleptsyn}, based on the Markov property. Choose some $t_0 > 0$, for instance $t_0 = 1$. Using lemma \ref{lem:saut-brownien}, choose $R > 0$ large enough such that with probability at least $1 - \eta$, for any starting point $p'$, a path $\gamma \in \tilde{\Gamma}_{p'}$ satisfies $\mathrm{diam}_g\left(\gamma([0,t_0])\right) < R/2$. In particular, $\nu_{\tilde{p}}^{t_0}(B_g(p,R)^{c}) \leq \eta$.
					
					Because the heat kernel is positive and smooth, we can consider a small neighborhood $U \subset B_g(\tilde{p},R/2)$ of $p$ in $\tilde{L}_p$ such that if $p' \in U$, for all $z \in B_g(p,R)$, we have $p_g(p', z;t_0) \geq (1 - \eta)p_g(p, z;t_0)$. 
					
					Now write $X_z \subset \tilde{\Gamma}_z$ for the set of paths starting from $z$ satisfying the conditions of definition \ref{def:Acontinu} shifted by $t_0$. 
					Then, if $\gamma \in \mathscr{A}_{\tilde{p}}(\kappa, B ; \beta, \rho, L)$ we have $\sigma_{t_0}(\gamma) \in X_{\gamma(t_0)}$ and so by the Markov property:
					\begin{align*}
					\tilde{W}_{\tilde{p}}(\mathscr{A}_{\tilde{p}}(\kappa, B ; \beta, \rho, L))
					&\leq \nu_{\tilde{p}}^{t_0}(B_g(\tilde{p},R)^{c}) + \int_{B_g(\tilde{p},R)} p_g(\tilde{p},z;t_0)\tilde{W}_z(X_z)\mathrm{d}\mathrm{vol}_g(z)\\
					&\leq \eta + \int_{B_g(\tilde{p},R)} p_g(\tilde{p},z;t_0)\tilde{W}_z(X_z)\mathrm{d}\mathrm{vol}_g(z).
					\end{align*}
					On the other hand, the function $\ell$ is bounded from above on $B_g(\tilde{p},R)$ by a constant $B'$, and if $z, z' \in B_g(p,R)$, the holonomy between $z$ and $z'$ is controlled by lemma \ref{lem:controle-hol2} by some large constant $L' > 0$ on the transverse section $S(z, \rho')$ for a certain $\rho' > 0$ which depends only on $R$ and the geometry of the foliation. Moreover we can assume that $\rho' < \rho$, $L' > L$ and $B' > B + \kappa\beta t_0$.
					Therefore, concatenating a continuous path from $p'$ to $z$ which remains in $B_g(\tilde{p},R)$ with a path in $X_z$ gives a path in $\mathscr{A}_{\tilde{p}}(\kappa, B' ; \beta, \rho', L')$ and so by the Markov property:\begin{align*}
					\tilde{W}_{p'}(\mathscr{A}_{p'}(\kappa, B' ; \beta, \rho', L'))
					&\geq  \tilde{W}_{p'}\left\{\gamma \in \tilde{\Gamma}_{p'}\,:\,\gamma([0,t_0]) \subset B_g(\tilde{p},R) \text{ and } \sigma_{t_0}(\gamma) \in X_{\gamma(t_0)} \right\}\\
					&\geq \int_{B_g(\tilde{p},R)} p_g(p',z;t_0)\tilde{W}_z(X_z)\mathrm{d}\mathrm{vol}_g(z) - \eta\\
					&\geq (1 - \eta)\int_{B_g(\tilde{p},R)} p_g(\tilde{p},z;t_0)\tilde{W}_z(X_z)\mathrm{d}\mathrm{vol}_g(z) - \eta\\
					&\geq \tilde{W}_{\tilde{p}}(\mathscr{A}_{\tilde{p}}(\kappa, B ; \beta, \rho, L)) - 3\eta,
					\end{align*}which proves the claim, since $\eta > 0$ can be chosen arbitrarily close to $0$.
				\end{proof}
				
				\begin{assumption}
					From now on, the parameters $\beta$ and $\kappa$ will be assumed to be small enough (in particular $\kappa < 1$) but bounded away from zero, and we fix a large constant $R  \geq \RLemSautBrownien$ depending only on the geometry of the foliation. The exact requirements on these constants will be defined implicitly during the proofs below. Later we will apply this in particular with $\beta = \alpha$.
					\label{assumption-kappa}
				\end{assumption}
				
				\begin{defn}There exists a constant $\CstDefnA > 0$ such that we can make the following definition:
					
					Let $p \in M \setminus S$ and $\tilde{p}$ be a lift in $\tilde{L}_p$. Let $\kappa, \beta \asymp 1$ and $L > 0, B > 1$. Let $\rho \in(0,\rhoSection)$. Let $\theta \in [0, \rho)$ such that $L\theta e^{\CstDefnA B} \lesssim 1$. Let $q \in S(p,\theta)$ and $\tilde{q}$ be a lift in $\tilde{L}_q$.
					
					Define $\mathsf{A}^{\delta}_{\tilde{p}\tilde{q}}(\kappa, B ; \beta, \rho, L)$ to be the set of paths $\gamma \in \tilde{\Gamma}_p$ such that for every integers $n, m$ with $n \geq m \geq 0$ and $i = 1, \ldots, K_n$, $j = 1, \ldots, K_m$ (the definition of the notations are given below for clarity):
					\begin{enumerate}
						\item $p_{n,i} \in E^{\delta/K_n}(p_{n,i-1}, q_{n, i-1}, R)$, \label{item:Aun}
						\item the holonomy between $p_{m,j}$ and $p_{n,i}$ maps $S(p_{m,j},\rho e^{-\kappa\beta t_{m,j}})$ to $S\left(p_{n,i},L\rho e^{-\beta (t_{n,i}-t_{m,j})}\right)$, and moreover, for any $z_0, z_1 \in S(p_{m,j},\rho e^{-\kappa\beta t_{m,j}})$ we have
						\[\frac{\dist\left(h_{p_{m,j}p_{n,i}}(z_0), h_{p_{m,j}p_{n,i}}(z_1)\right)}{\dist(\gamma(t_{n,i}),S)} \leq Le^{-\beta(t_{n,i}-t_{m,j})}e^{\kappa\beta t_{m,j}}\frac{\dist(z_0, z_1)}{\dist(\gamma(t_{m,j}),S)}; \]
						\item $\ell(p_{n,i}) \leq \kappa \beta \delta n + B$.
						
					\end{enumerate}
				Here we have written $K_n$ for the superior integer part of $\kappa\beta\delta n + B$ for any $n$, $t_{n,i } = n\delta + i\delta/K_n$, $p_{n,i} = \gamma(t_{n,i})$, and the sequence $(q_{n,i})$ is defined inductively by $q_{0,0} = \tilde{q}$, $q_{n,i} = \Phi^{-1}_{q_{n,i-1}p_{n,i-1}}\left(p_{n,i}\right)$, and $q_{n,K_n} = q_{n + 1,0}$, $p_{n,K_n} = p_{n + 1,0}$. That is, $q_{n,i}$ is the image of $q_{0,0}$ by the holonomy from $p_{0,0}$ to $p_{n,i}$. Finally, we have again abused notation and written $h_{p_{m,j}p_{n,i}}$ for the holonomy between the projections to $L_p$ of $p_{m,j}$ and $p_{n,i}$.
					\label{def:A}
				\end{defn}
				\begin{rem}
					\begin{enumerate}
						\item Note that the projection is well-defined since for every $n \geq 0$ and $i \geq 1$,
						\begin{align*}\exp\left(\CstLemProjGlobale\ell(p_{n,i-1})e^{2\CstLemProjGlobale R}\right)\dist(p_{n,i-1}, q_{n,i-1})
						&\leq \exp\left(\CstLemProjGlobale K_ne^{2\CstLemProjGlobale R}\right)L\theta e^{-\beta t_{n,i-1}}\dist(p_{n,i},S) \\
						&\leq \mathrm{diam}(M) L\theta e^{\CstDefnA(B + 1)}\\
						&< \epsLemProjLocale,\end{align*}
						if we take $\CstDefnA > \CstLemProjGlobale e^{2\CstLemProjGlobale R}$, and so the image of $\Phi_{q_{n,i-1}p_{n,i-1}}$ contains $B_g(p_{n,i-1}, R)$ by lemma \ref{lem:proj-globale}, and this set contains $p_{n,i}$, since by definition $E^{\delta/K_n}(p_{n,i-1}, q_{n,i-1}, R) \subset B_g(p_{n,i-1}, R)$. For the same reason, the condition \ref{eq:comparaison} in proposition \ref{lem:comparaison} is satisfied for the appropriate choice of constants.
						\item The idea of definition \ref{def:A} is the same as that of definition \ref{def:Acontinu}, except that:
						\begin{enumerate}
							\item for technical reasons, we only impose conditions on the path at a certain set of discretised times $t_{n,i}$ ;
							\item we impose contraints at those times that the points belong to certain sets given by proposition \ref{lem:comparaison} where we can compare the heat kernels on different leaves.
						\end{enumerate}
					\end{enumerate}
				\end{rem}
				\begin{lem}There exist constants $\CstLemProbaA, \cstLemProbaA > 0$ such that the following holds:
					
					Let $\delta \in (0,1)$. Let $p \in M \setminus S$ and $\tilde{p}$ be a lift in $\tilde{L}_p$. Let $\kappa, \beta \asymp 1$ and $L > 0, B > 1$. Let $\rho \in(0,\rhoSection)$. Let $\theta \in [0, \rho)$ such that $L\theta e^{\CstDefnA B} \lesssim 1$. Let $q \in S(p,\theta)$ and $\tilde{q}$ be a lift in $\tilde{L}_q$.
					Then we have
					\[ \tilde{W}_{\tilde{p}}\left(\mathscr{A}_{\tilde{p}}(\kappa, B ; \beta, \rho, L) \setminus \mathsf{A}^{\delta}_{\tilde{p}\tilde{q}}(\kappa, B ; \beta, \rho, L) \right)<\delta^{-\CstLemProbaA}L\theta e^{\CstLemProbaA B} + \CstLemProbaA\exp\left(-\cstLemProbaA/\delta\right).\]
					\label{lem:proba-A}
				\end{lem}
				\begin{proof}
					We follow Deroin-Kleptsyn \cite[end of proof of lemma 5.2]{deroinkleptsyn}. Let $N \geq 0$ be an integer and $I$ an integer with $I < K_N$, where $K_N$ is as usual the superior integer part of $\ell_N := \kappa\beta\delta N + B$. Consider $X_{N,I}$ the set of paths $\gamma \in \mathscr{A}_{p}(\kappa, B ; \beta, \rho, L)$  such that the condition \ref{item:Aun} in definition \ref{def:A} is satisfied for $n \leq N$ and $i \leq I$.
					
					Then, using the notations of definition \ref{def:A}, the measure $\mathscr{W}'_{N,I}$ which is the pushforward of $\restr{\tilde{W}_{\tilde{p}}}{X_{N, I + 1}}$ by the map $\gamma \mapsto (p_{n,i})_{n \leq N, i \leq I}$ is absolutely continuous with respect to the measure $\mathscr{W}_{N,I}$ which is the pushforward of $\restr{\tilde{W}_{\tilde{p}}}{X_{N,I}}$ by the map $\gamma \mapsto (p_{n,i})_{n \leq N, i \leq I}$, and the density is \begin{align*}(p_{n,i}) \mapsto &\nu_{p_{N,I}}^{\delta/K_N}(E^{\delta/K_N}(p_{N,I}, q_{N,I}, R)) \\
					&\leq \CstBLemComparaison \left(\frac{\delta}{K_N}\right)^{-\CstBLemComparaison}\left(\exp\left(\CstBLemComparaison\ell_Ne^{\CstBLemComparaison R}\right)\dist(p_{N,I}, q_{N,I})  + \exp\left(-\cstLemComparaison K_N\frac{R^2}{\delta}\right)\right)^{1/2}\\
					&\leq C \left(\frac{\delta}{K_N}\right)^{-\CstBLemComparaison}\left(\exp\left(\CstBLemComparaison\ell_Ne^{\CstBLemComparaison R}\right)L\theta e^{-\beta\delta N}  + \exp\left(-\cstLemComparaison K_N\frac{R^2}{\delta}\right)\right)^{1/2},\end{align*} where $q_{N,I}$ is the image of $q$ by the holonomy along $\gamma$ from $p$ to $p_{N,I}$.
					
					Order the set $\mathscr{E}$ of pairs $(N,I)$ with $0 \leq I < K_N$ in lexicographic order, that is in such a way that the function $(N, I) \mapsto t_{N,I}$ (see definition \ref{def:A}) is increasing. Let $(c_m)_{m \in \N}$ be an enumeration of $\mathscr{E}$ in increasing order. Then we have $\mathsf{A}^{\delta}_{pq}(\kappa, B ; \beta, \rho, L) = \bigcap_m \left(X_{c_m}\setminus X_{c_{m + 1}}\right)$ and so we get:
					\begin{align*}
					\tilde{W}_{\tilde{p}}\left(\mathscr{A}_{\tilde{p}}(\kappa, B ; \beta, \rho, L) \setminus \mathsf{A}^{\delta}_{pq}(\kappa, B ; \beta, \rho, L) \right) 
					= \sum_{m} \tilde{W}_{\tilde{p}}\left(X_{c_m}\setminus X_{c_{m + 1}}\right)\\
					\leq \sum_{n = 0}^\infty CK_n \left(\frac{\delta}{K_n}\right)^{-\CstBLemComparaison}\left(\exp\left(\CstBLemComparaison\ell_ne^{\CstBLemComparaison R}\right)L\theta e^{-\beta\delta n }  + \exp\left(-\cstLemComparaison K_n\frac{R^2}{\delta}\right)\right)^{1/2}.
					\label{eq:grosse serie}
					\end{align*}
					We now bound this last sum. First note that \begin{align*}
					K_n^{\CstBLemComparaison + 1}\exp\left(\CstBLemComparaison\ell_ne^{\CstBLemComparaison R}\right) 
					&\leq C(\kappa\beta\delta n + B)^{\CstBLemComparaison + 1}\exp\left(\CstBLemComparaison\kappa\beta\delta e^{\CstBLemComparaison R}n\right)\exp\left(\CstBLemComparaison e^{\CstBLemComparaison R}B\right) \\
					&\leq CB^{C}\exp\left(\left(\CstBLemComparaison + 1 + \CstBLemComparaison e^{\CstBLemComparaison R}\right)\kappa\beta\delta n\right)\exp\left(CB\right)\\
					&\leq \exp\left(CB\right)\exp\left(\frac{\beta\delta}{2} n\right),
					\end{align*}
					since $\kappa \lesssim 1$. Therefore
					\[C K_n\left(\frac{\delta}{K_n}\right)^{-\CstBLemComparaison}\left(\exp\left(\CstBLemComparaison\ell_ne^{\CstBLemComparaison R}\right)e^{-\beta\delta n}\right)^{1/2}
					\leq  \delta^{-C}e^{CB}L\theta \exp\left(-\frac{\beta\delta}{4} n\right).\]
					For the second part, we have
					\begin{align*}
					K_n^{\CstBLemComparaison + 1}\exp\left(-\cstLemComparaison K_n\frac{R^2}{2\delta}\right)&\leq CB^{C}\exp\left(\left(\CstBLemComparaison + 1\right)\kappa\beta\delta n\right)\exp\left(-\cstLemComparaison (\kappa\beta\delta n + B)\frac{R^2}{2\delta}\right)\\
					&\leq C\exp\left(-cB/\delta\right)\exp(-\kappa\beta\delta n)
					\end{align*}
					for $R$ large enough, depending only on the geometry of the foliation. Therefore, summing the series, the right-hand side of equation \ref{eq:serie} is bounded above by $\delta^{-C}L\theta e^{CB} + C\exp\left(-cB/\delta\right)$, which proves the claim.
				\end{proof}
				
				\begin{defn}
					Let $p \in M \setminus S$ and $\tilde{p}$ be a lift in $\tilde{L}_p$. Let $\kappa, \beta \asymp 1$ and $L > 0, B > 1$. Let $\rho \in(0,\rhoSection)$. Let $\theta \in [0, \rho)$ such that $L\theta e^{\CstDefnA B} \lesssim 1$. Let $q \in S(p,\theta)$ and $\tilde{q}$ be a lift in $\tilde{L}_q$.
					
					Define $\mathsf{B}^{\delta}_{\tilde{p}\tilde{q}}(\kappa, B ; \beta, \rho, L)$ to be the set of paths $\omega \in \tilde{\Gamma}_{\tilde{q}}$ such that for every integers $n, m$ with $n \geq m \geq 0$ and $i = 1, \ldots, K_n$, $j = 1, \ldots, K_m$:
					\begin{enumerate}
						\item $q_{n,i}\in \Phi_{q_{n,i-1}p_{n,i-1}}^{-1}\left(E^{\delta/K_n}(p_{n,i-1}, q_{n, i-1}, R)\right)$,\label{def:B-cond}
						\item the holonomy between $p_{m,j}$ and $p_{n,i}$ maps $S(p_{m,j},\rho e^{\kappa\beta t_{m,j}})$ to $S\left(p_{n,i},L\rho e^{-\beta (t_{n,i}-t_{m,j})}\right)$, and moreover, for any $z_0, z_1 \in S(p_{m,j},\rho e^{-\kappa\beta t_{m,j}})$ we have
						\[\frac{\dist\left(h_{p_{m,j}p_{n,i}}(z_0), h_{p_{m,j}p_{n,i}}(z_1)\right)}{\dist(p_{n,i},S)} \leq Le^{-\beta(t_{n,i}-t_{m,j})}e^{\kappa\beta t_{m,j}}\frac{\dist(z_0, z_1)}{\dist(p_{m,j},S)}; \]
						\item $\ell(p_{n,i}) \leq \kappa\beta \delta n + B$.
					\end{enumerate}
					Here we have written $K_n$ for the superior integer part of $\kappa\beta\delta n + B$ for any $n$, $t_{n,i} = n\delta + i\delta/K_n$, $q_{n,i} = \omega(t_{n,i})$, $q_{n,K_n} = q_{n + 1,0}$, $p_{n,i} = \Phi_{q_{n,i-1}p_{n,i-1}}(q_{n,i})$, $p_{n,K_n} = p_{n + 1,0}$, $p_{0,0} = \tilde{p}$. As above $h_{p_{m,j}p_{n,i}}$ is an abuse of notation for the holonomy between the projections of $p_{m,j}$ and $p_{n,i}$ on $L_p$.
					\label{def:B}
				\end{defn}
				\begin{rem}
					\begin{enumerate}
						\item The set $\Phi_{q_{n,i-1}p_{n,i-1}}^{-1}\left(E^{\delta/K_n}(p_{n,i-1}, q_{n, i-1}, R)\right)$ is well-defined for the same reason as in definition \ref{def:A} above.
						\item The set $\mathsf{B}^{\delta}_{\tilde{p}\tilde{q}}(\kappa, B ; \beta, \rho, L)$ is the counterpart of $\mathsf{A}^{\delta}_{\tilde{p}\tilde{q}}(\kappa, B ; \beta, \rho, L)$ when a path is pushed orthogonally and "seen from" paths starting at $\tilde{q}$.
					\end{enumerate}
				\end{rem}

				\begin{lem}There exists constants $\CstLemAbsoluteContinuite, \cstLemAbsoluteContinuite > 0$ such that the following holds:	
					
					Let $p \in M \setminus S$ and $\tilde{p}$ be a lift in $\tilde{L}_p$. Let $\kappa, \beta \asymp 1$ and $L > 0, B > 1$. Let $\rho \in(0,\rhoSection)$. Let $\theta \in [0, \rho)$ such that $L\theta e^{\CstDefnA B} \lesssim 1$. Let $q \in S(p,\theta)$ and $\tilde{p}$ be a lift in $\tilde{L}_p$.
					
					Consider $\mu_A$ the image measure of $\restr{\tilde{W}_{\tilde{p}}}{\mathsf{A}^{\delta}_{\tilde{p}\tilde{q}}(\kappa, B ; \beta, \rho, L)}$ by the map $\gamma \mapsto (p_{n,i})_{n\geq 0,0\leq i < K_n}$, in the notation of the definition \ref{def:A}. Similarly consider $\mu_B$ the image measure of $\restr{\tilde{W}_{\tilde{q}}}{\mathsf{B}^{\delta}_{\tilde{p}\tilde{q}}(\kappa, B ; \beta, \rho, L)}$ by the map $\gamma \mapsto (p_{n,i})_{n \geq 0,0\leq i < K_n}$, in the notation of the definition \ref{def:B}.
					
					Then the measure $\mu_A$ is absolutely continuous with respect to $\mu_B$. Moreover, if we write 
					\[\eta := \delta^{-\CstLemAbsoluteContinuite}L\theta e^{\CstLemAbsoluteContinuite B} + \CstLemAbsoluteContinuite\exp\left(-\cstLemAbsoluteContinuite B/\delta\right),\] the density of $\mu_A$ with respect to $\mu_B$ has values in $[1 -\eta, 1 + \eta]$.
					
					In particular, we have $\tilde{W}_{\tilde{q}}\left(\mathsf{B}^{\delta}_{\tilde{p}\tilde{q}}(\kappa, B ; \beta, \rho, L)\right) \geq \tilde{W}_{\tilde{p}}\left(\mathsf{A}^{\delta}_{\tilde{p}\tilde{q}}(\kappa, B ; \beta, \rho, L)\right) - \eta$.
					\label{lem:absolue-continuite}
				\end{lem}
				\begin{proof}
					Let $N \in \N$. Consider $\mu_A^N$ the image measure of $\restr{\tilde{W}_{\tilde{p}}}{\mathsf{A}^{\delta}_{\tilde{p}\tilde{q}}(\kappa, B ; \beta, \rho, L)}$ by the map $\gamma \mapsto (p_{n,i})_{0\leq n\leq N,0\leq i < K_n}$, in the notation of the definition \ref{def:A}. Similarly consider $\mu_B^N$ the image measure of $\restr{\tilde{W}_{\tilde{q}}}{\mathsf{B}^{\delta}_{pq}(\kappa, B ; \beta, \rho, L)}$ by the map $\gamma \mapsto (p_{n,i})_{0\leq n\leq N,0\leq i < K_n}$, in the notation of the definition \ref{def:B}.
					For a path in $\mathsf{A}^{\delta}_{\tilde{p}\tilde{q}}(\kappa, B ; \beta,  \rho, L)$, we have $q_{n,i} \in B_g(q_{n,i-1},2R)$, so by proposition \ref{lem:comparaison}, the measure $\mu_A^N$ is absolutely continuous with respect to $\mu_B^N$, and the density $\rho_N$ can be written as:
					\begin{equation}
					\rho_N := \prod_{n = 0}^{N}\prod_{i = 0}^{K_n-1}\frac{\mathrm{d}\restr{\nu_{p_{n,i}}^\delta}{B_g(p_{n,i}, R)}}{\mathrm{d}\left(\left(\Phi_{p_{n,i}q_{n,i}}\right)_{*}\restr{\nu_{q_{n,i}}^\delta}{B_g(q_{n,i}, 2R)}\right)}(p_{n,i}). 
					\label{eq:densite-prod}\end{equation}
					By lemma \ref{lem:comparaison}, we have
					\begin{equation*}
					\abs{\log \rho_N} \leq  \sum_{n = 0}^N C K_n\left(\frac{\delta}{K_n}\right)^{-\CstBLemComparaison}\left(\exp\left(\CstBLemComparaison\ell_ne^{\CstBLemComparaison R}\right)\sup_i \dist(p_{n,i}, q_{n,i})  + \exp\left(-\cstLemComparaison K_n\frac{R^2}{\delta}\right)\right)^{1/2},
					\label{eq:serie}
					\end{equation*}
					where $\ell_n = \kappa\beta\delta n + B$ and $K_n$ is the superior integer part of $\ell_n$, for every $n \in \N$. This sum can be bounded in the exact same way as in the proof of lemma \ref{lem:proba-A}, by $\delta^{-C}L\theta e^{CB} + C\exp\left(-cB/\delta\right)$.

					Therefore, by a simple limiting argument, the measure $\mu_A$ is absolutely continuous with respect to $\mu_B$ and the density is the limit as $N \longrightarrow \infty$ of $\rho_N$.
				\end{proof}
				
				\begin{coro}
					Using the notations of lemma \ref{lem:absolue-continuite}, assume that $\eta < 1$. Then the measure $\mu_B$ is absolutely continuous with respect to $\mu_A$. In particular, if $p$ is such that for $W_p$-almost every $\gamma$, we have \[\frac{1}{n}\sum_{j = 0}^{n - 1} \delta_{\gamma(j\delta)} \rightharpoonup \nu,\]
					then for $\tilde{W}_{\tilde{q}}$-almost every $\omega \in \mathsf{B}^{\delta}_{pq}(\kappa, B ; \beta, \rho, L)$, we have
					\[\frac{1}{n}\sum_{j = 0}^{n - 1} \delta_{\pi_p(p_{j,0})} \rightharpoonup \nu,\]
					where we use the notation of definition \ref{def:B}.
					\label{coro:genericite}
				\end{coro}
				\begin{lem}There exists a constant $\CstLemRedDelta > 0$ such that the following holds:	
					
					Let $p \in M \setminus S$ and $\tilde{p}$ a lift of $p$ in $\tilde{L}_p$. Let $\kappa, \beta \asymp 1$ and $L > 0, B > 1$. Let $\rho \in(0,\rhoSection)$. Let $\theta_0 \in [0, \rho)$ such that $L\theta_0e^{\CstDefnA B} \lesssim 1$. Let $q \in S(p,\theta_0)$ and $\tilde{q}$ a lift of $q$ in $\tilde{L}_q$.
					
					Then for $\tilde{W}_{\tilde{q}}$-almost every $\omega \in \mathsf{B}^{\delta}_{\tilde{p}\tilde{q}}(\kappa, B ; \beta,\rho, L)$, we have $\omega \in \mathscr{A}_{\tilde{q}}(\CstLemRedDelta\kappa; \beta/2)$.
					\label{lem:dediscretisation}
				\end{lem}
				\begin{proof}Let $\omega \in \mathsf{B}^{\delta}_{\tilde{p}\tilde{q}}(\kappa, B ; \beta,\rho, L)$.
					By the Borel-Cantelli lemma and lemma \ref{lem:saut-brownien}, $\tilde{W}_{\tilde{q}}$-almost surely, there exists $n_0 \geq 0$ such that for any integers $n \geq n_0$ and $i$ with $1 \leq i \leq K_n$, for every $s \in[t_{n,i-1},t_{n,i}]$ we have $\dist_g(\omega(t_{n,i-1}), \omega(s)) \leq R$. 
					
					Assume that $n \geq n_0$ and $0 \leq i < K_n$. We can inductively define the continuous path $\gamma $ on $[n_0\delta, \infty)$ which coincides with $\Phi_{q_{n,i-1}p_{n,i-1}} \circ \omega$ on $[t_{n,i-1},t_{n,i}]$. The projection is well-defined for the same reason as above (see the remark below definition \ref{def:A}). 
					
					Now for any $t \in[t_{n,i-1},t_{n,i}]$, because the projection is $2$-Lipschitz, we have $\dist_g(\gamma(t_{n,i-1}), \gamma(t)) \leq 2R$. Therefore, \[\ell(\gamma(t)) \leq C\ell(p_{n,i}) \leq C\kappa\beta\delta n + CB \leq C\kappa\beta t + CB.\] 

					Let now $t \geq s \geq 0$, with $t \in[t_{n,i-1},t_{n,i}]$ where $n,i$ are as above. We assume that $t \geq n_0\delta$ and $s \in [n_0\delta,t]$ and let $m,j$ be integers as above such that $s \in [t_{m,j-1}, t_{m,j}]$. Let $\rho' > 0$ to be chosen later and $z_0, z_1 \in S(\gamma(s),\rho')$ and denote by $z_0(\tau), z_1(\tau)$ the respective images of $z_0, z_1$ by the holonomy along (the projection in $L_p$ of) $\gamma$ from times $s$ to $\tau$, for any real number $\tau \geq s$. Let \[\rho'(\tau) := \max\left\{ \frac{\dist(z_0(\tau),\gamma(\tau))}{\dist(\gamma(\tau),S)}, \frac{\dist(z_1(\tau),\gamma(\tau))}{\dist(\gamma(\tau),S)}\right\},\]\[\theta(\tau) := \frac{\dist(z_0(\tau),z_1(\tau))}{\dist(\gamma(\tau),S)}.\]
					Then assuming that $\rho' \leq L^{-1}\rho\exp\left(-C\kappa\beta s-CB\right)$, we have 
					\begin{align*}
					C\exp\left(C\ell(\gamma(t_{m,j}))e^{2CR}\right)\rho'
					&\leq e^{CB}\exp\left(C\kappa\beta s\right)\rho'\\
					&\lesssim 1,
					\end{align*}and thus by lemma \ref{lem:controle-hol2},\begin{align*}
					\theta(t_{m,j})
					&\leq C\exp\left(C\ell(\gamma(t_{m,j}))e^{2CR}\right)\theta(s)\\
					&\leq e^{CB}\exp\left(C\kappa\beta s\right)\theta(s),
					\end{align*}
					\begin{align*}
					\rho'(t_{m,j})
					\leq e^{CB}\exp\left(C\kappa\beta s\right)\rho'.
					\end{align*}
					Because of the constraint on $\rho'$, changing the constant $C$ above if necessary, we have therefore $\rho'_a(t_{m,j}) \leq \rho e^{-\kappa\beta t_{m,j}}$ and so
					\begin{align*}
					\rho'_a(t_{n,i-1})
					&\leq Le^{-\beta(t_{n,i-1}-t_{m,j})}e^{\kappa\beta t_{m,j}}\rho'_a(t_{m,j})\\
					&\leq Le^{CB}Le^{-\beta(t-s)}e^{C\kappa\beta s}\rho'\\
					&\leq Le^{CB+C\kappa\beta n_0\delta}e^{-\beta(t-s)}e^{C\kappa\beta(s-n_0\delta)}\rho',
					\end{align*}
					\begin{align*}
					\theta(t_{n,i-1})
					&\leq Le^{CB+C\kappa\beta n_0\delta}e^{-\beta(t-s)}e^{C\kappa\beta(s-n_0\delta)}\theta(s).\\
					\end{align*}
					Since $\rho' \leq L^{-1}\rho\exp\left(-C\kappa\beta s-CB\right)$, we have 
					\begin{align*}C\exp\left(C\ell(\gamma(t_{n,i-1}))e^{2CR}\right)\rho'(t_{n,i-1})
					&\leq Le^{CB+C\kappa\beta n_0\delta}e^{-\beta(1 - C\kappa)(t-s)}e^{C\kappa\beta s}\rho'(s)\\
					&\lesssim 1.
					\end{align*}
					Therefore we can apply lemma \ref{lem:controle-hol2} again:
					\begin{align*}\theta(t)
					&\leq C\exp\left(C\ell(\gamma(t_{n,i-1}))e^{2CR}\right)\theta(t_{n,i-1})\\
					&\leq Le^{CB+C\kappa\beta n_0\delta}e^{-\beta(1 - C\kappa)(t-s)}e^{C\kappa\beta s}\theta(s),
					\end{align*}
					\begin{align*}\rho'(t) \leq Le^{CB+C\kappa\beta n_0\delta}e^{-\beta(1 - C\kappa)(t-s)}e^{C\kappa\beta s}\rho'(s).\end{align*}
					We assume, as we may, that $C\kappa < 1/2$. Then writing \[\kappa' := C\kappa,\] \[B' := CB,\] 
					\[\beta' : = (1-C\kappa)\beta > \beta/2,\]
					\[\rho' := L^{-1}e^{-CB-C\kappa\beta n_0\delta}\rho,\]
					\[L' := Le^{CB+C\kappa\beta n_0\delta},\]
					we have $\sigma_{n_0\delta}(\gamma) \in \mathscr{A}_{\gamma(n_0\delta)}(\kappa', B' ; \beta', \rho', L')$.
					
					For $n_1$ large enough, we have
					\begin{align*}
					\frac{\dist(\omega(n_1\delta)),\gamma(n_1\delta))}{\dist(\gamma(n_1\delta), S)} 
					&\leq Le^{-\beta n_1\delta}\theta_0\\
					&\leq \rho'/C.
					\end{align*}
					Therefore, by lemma \ref{lem:changer-transversale}, we can deduce that $\sigma_{n_1\delta}(\omega) \in \mathscr{A}_{\omega(n_1\delta)}(\kappa', \beta')$ if $n_1$ is large enough, and therefore using lemma \ref{lem:controle-hol2} again we deduce that $\omega \in \mathscr{A}_{\tilde{q}}(\kappa', \beta')$.

				\end{proof}
				\begin{lem}There exists a constant $\CstLemRedDelta > 0$ such that the following holds:	
					
					Let $p \in M \setminus S$. Let $\kappa, \beta \asymp 1$ and $L > 0, B > 1$. Let $\rho \in(0,\rhoSection)$. Let $\theta > 0$ be such that $L\theta e^{\CstLemRedDelta B} \lesssim 1$ and $\theta < \rho$ and let $q \in S(p,\theta)$.
					
					Assume that for all $\delta \in (0,1) \cap \Q$, we have for $W_p$-almost every path $\gamma$,
					\begin{equation}\frac{1}{n}\sum_{j = 0}^{n - 1} \delta_{\gamma(j\delta)} \rightharpoonup \nu.
					\label{eq:gen1}\end{equation}
					
					Let $\delta' \in (0,1)\cap\Q$. Then for $\tilde{W}_{\tilde{q}}$-almost every $\omega \in \mathsf{B}^{\delta}_{pq}(\kappa, B ; \beta,\rho, L)$, there exists $M_0 > 0$ and a sequence $(p_n)$ of points in $\tilde{L}_p$ such that $\dist(\omega(n\delta'), p_n) \leq M_0e^{-\beta\delta n/2}$, $\ell(p_n) \leq \CstLemRedDelta\kappa\beta\delta n + M_0$ and moreover
					\begin{equation}\frac{1}{n}\sum_{j = 0}^{n - 1} \delta_{\pi_p(p_j)} \rightharpoonup \nu.\label{eq:distribution-p}\end{equation}
					\label{lem:red-delta}
				\end{lem}
				\begin{proof}
					Note that condition \ref{eq:gen1} is a tail-type property which hold $W_p$-almost surely, so they hold $W_{p'}$-almost surely for every $p' \in L_p$ (see lemma \ref{lem:queue}). 
					
					Using the constructions and notations of the proof of lemma \ref{lem:dediscretisation}, the proof follows by an argument very close to that of lemma \ref{lem:proba-A}, using the Borel-Cantelli lemma to ensure that all but a finite number of the conditions \ref{def:B-cond} are satisfied (for a discretisation step $\delta'$), if $\kappa \lesssim 1$. Moreover, using corollary \ref{coro:genericite}, we can ensure that equation \ref{eq:distribution-p} is satisfied, since writing \[\theta_0' := \frac{\dist(\omega(n_1\delta)),\gamma(n_1\delta))}{\dist(\gamma(n_1\delta), S)},\] we have
					$\left(\delta'\right)^{-C}L'\theta_0'e^{CB'} + Ce^{-cB'/\delta'} < 1/2$ providing $n_0$ is large enough and adjusting the constants if necessary.
				\end{proof}
				Using lemma \ref{lem:saut-brownien}, we can prove the following lemma (see also Deroin-Kleptsyn \cite[subsection 5.1]{deroinkleptsyn} for the idea of proof):
				\begin{lem}
					Let $\eps > 0$. Let $r > 0$. Let $q \in M\setminus S$ and $\tilde{q}$ a lift of $q$ in $\tilde{L}_q$. Then for every $\delta > 0$ small enough, $\tilde{W}_{\tilde{q}}$-almost every path $\omega \in \tilde{\Gamma}_{\tilde{q}}$ is such that the upper density of the set
					\[\left\{n \in \N\,:\,\diam\left(\gamma([n\delta, (n + 1)\delta])\right) > r\right\}\]
					is smaller than $\eps$.
					\label{lem:densite-sauts}
				\end{lem}
				
				Now we can state and prove the main proposition of this section:
				\begin{prop}Let $\kappa,\beta \asymp 1$
					Let $\eta > 0$.
					Let $p \in P \setminus S$ be a a point such that for $W_p$-almost every path $\gamma \in \Gamma_p$, we have \[\frac{1}{n}\sum_{j = 0}^{n - 1} \delta_{\gamma(j\delta)} \rightharpoonup \nu\] 
					for all $\delta \in (0,1) \cap \Q$. Let $\tilde{p}$ be a lift of $p$ in $\tilde{L}_p$. Then there exists an open neighbourhood $U$ of $p$ such that the following holds, for $q\in U$, $\tilde{q}$ a lift of $q$ in $\tilde{L}_q$:
					\[\tilde{W}_{\tilde{q}}\left\{\omega \in \tilde{\Gamma}_{\tilde{q}}\,:\,\frac{1}{T}\left(\pi_q \circ \omega\right)_{*}\mathrm{Leb}_{[0,T]} \underset{T\to\infty}{\rightharpoonup} \nu\right\} \geq W_p(\mathscr{A}_p(\kappa,\beta)) - \eta.\]
					
					\label{lem:similarite}
				\end{prop}
				\begin{rem}
					Here the convergence is vague convergence on $M$. We will see later that using Nguyên's theorem (theorem \ref{thm:nguyen}), this can be improved to vague convergence in $M\setminus S$.
				\end{rem}
				By the Birkhoff ergodic theorem, Nguyên's theorem, and Fornæss-Sibony's theorem (theorem \ref{thm:fornaess-sibony}), and proposition \ref{prop:contraction}, we deduce:
				\begin{coro}
					Let $\eta > 0$.
					Let $p \in P \setminus S$ be a $\nu$-generic point. There exists an open neighbourhood $U$ of $p$ such that the following holds, for $q\in U$, $\tilde{q}$ a lift of $q$ in $\tilde{L}_q$:
					\[\tilde{W}_{\tilde{q}}\left\{\omega \in \tilde{\Gamma}_{\tilde{q}}\,:\,\frac{1}{T}\left(\pi_q \circ \omega\right)_{*}\mathrm{Leb}_{[0,T]} \underset{T\to\infty}{\rightharpoonup} \nu\right\} \geq 1 - \eta.\]
					
					\label{coro:similarite}
				\end{coro}
				\begin{proof}[Proof of proposition \ref{lem:similarite}]
					
					Choosing $B$ and $L$ large enough, and $\rho > 0$ small enough, we can assume that $\tilde{W}_{\tilde{p}}\left(\mathscr{A}_p(\kappa, B;\beta, \rho, L)\right) > \tilde{W}_{\tilde{p}}\left(\mathscr{A}_{\tilde{p}}(\kappa,\beta)\right) - \eta$. By applying lemma \ref{lem:queue-A} and further increasing $L$ and $B$, and reducing $\rho$, we have this property replacing $\tilde{p}$ by $p'$ in a small neighborhood $V$ of $\tilde{p}$ in $\tilde{L}_p$.
					
					Choose $\delta \in (0,1) \cap \Q$ fixed, for instance $\delta = 1/2$. If $B$ is large enough, which we can assume, by lemma \ref{lem:proba-A}, that there is a neighborhood $U$ of $p$ in $M$ such that if $q \in U$, there is $p' \in V$ (obtained by projecting on the leaf of $p$ and with lift $\tilde{p}'$) with $\tilde{W}_{\tilde{p}'}\left(\mathsf{A}^{\delta}_{\tilde{p}'\tilde{q}}(\kappa, B ; \beta, \rho,L)\right) > \tilde{W}_{\tilde{p}}\left(\mathscr{A}_{\tilde{p}}(\kappa, B;\beta, \rho, L)\right) - \eta > \tilde{W}_{\tilde{p}}\left(\mathscr{A}_{\tilde{p}}(\kappa, B;\beta, \rho, L)\right) - 2\eta$.	By lemma \ref{lem:absolue-continuite}, reducing $U$ if necessary, we also have if $q \in U$, $\tilde{W}_{\tilde{q}}\left(\mathsf{B}^{\delta}_{\tilde{p}\tilde{q}}(\kappa, B ; \beta, \rho, L)\right) \geq \tilde{W}_{\tilde{p}}\left(\mathsf{A}^{\delta}_{\tilde{p}\tilde{q}}(\kappa, B ; \beta, \rho, L)\right) - \eta > \tilde{W}_{\tilde{p}}\left(\mathscr{A}_{\tilde{p}}(\kappa,\beta)\right) - 3\eta$. We are going to prove that the paths of $\mathsf{B}^{\delta}_{\tilde{p}\tilde{q}}(\kappa, B ; \beta, \rho, L)$ are distributed according to $\mu$, which will establish the result.
					
					Let $f \in \mathrm{C}^0(M)$. Let $\eps > 0$. Choose $r > 0$ small enough such that for every $x,y$ with $\dist(x,y) < r$, we have $\abs{f(x) - f(y)} < \eps$. Choose $\delta' \in (0,1)\cap\Q$ small enough as per lemma \ref{lem:densite-sauts} with $r$ and $\eps/(2\norm{f}_\infty)$. 
					
					
					Now let $q \in U$, with lift $\tilde{q}$. By lemma \ref{lem:red-delta}, for $\tilde{W}_{\tilde{q}}$-almost every path $\omega \in \mathsf{B}^{\delta}_{\tilde{p}\tilde{q}}(\kappa, B ; \beta, \rho, L)$ there exists $M > 0$ and a sequence $(p_j) \in \tilde{L_p}$ such that $\dist(\omega(j\delta'), p_j) \leq Me^{-\beta \delta j/2}$ and equation \ref{eq:distribution-p} is satisfied. 
					Now we write:
					\begin{align*}&\abs{\frac{1}{n\delta'}\int_0^{n\delta'} f(\pi_q \circ \omega(t))\deriv t - \frac{1}{n}\sum_{j = 0}^{n - 1} f(\pi_q \circ \omega(j\delta)))} \\
					&\leq \frac{1}{n}\sum_{j = 0}^{n - 1}\abs{\frac{1}{\delta'} \int_{j\delta'}^{(j + 1)\delta'} f(\pi_q \circ \omega(t))\deriv t - f(\pi_q \circ \omega(j\delta'))}\\
					&\leq \eps + \frac{2\norm{f}_\infty}{n}\abs{\{j\in\{0,\ldots,n-1\}\,:\, \mathrm{diam}(\omega([j\delta', (j+1)\delta']) \geq r ) \}}
					\end{align*}Therefore:
					\[\limsup_{n\to\infty} \abs{\frac{1}{n\delta'}\int_0^{n\delta'} f(\omega(t))\deriv t - \frac{1}{n}\sum_{j = 0}^{n - 1} f(\omega(j\delta)))} \leq 2\eps.\]
					Because of equation \ref{eq:distribution-p}, we have for $n$ large enough
					\begin{align*}\abs{\frac{1}{n}\sum_{j = 0}^{n - 1} f(\pi_q \circ \omega(j\delta))) - \int f\deriv\nu} < \eps. \end{align*}
					Therefore we have for $\tilde{W}_{\tilde{q}}$-almost every path $\omega \in \mathsf{B}^{\delta}_{\tilde{p}\tilde{q}}(\kappa, B ; \beta, \rho, L)$:
					\begin{align*}\limsup_{n \rightarrow\infty} \abs{\frac{1}{n\delta'}\int_0^{n\delta'} f(\pi_q \circ \omega(t))\deriv t - \int f \deriv \nu} < 3\eps,
					\end{align*}
					Because $\eps$ is arbitrary, we have for $W_q$-almost every path $\omega \in \mathsf{B}^{\delta}_{pq}(\kappa, B ; \beta, L)$:
					\begin{align*}\limsup_{n \rightarrow\infty} \abs{\frac{1}{n\delta'}\int_0^{n\delta'} f(\pi_q \circ \omega(t))\deriv t - \int f \deriv \nu} = 0,
					\end{align*}
					and therefore
					\begin{align*}\limsup_{T \rightarrow\infty} \abs{\frac{1}{T}\int_0^{T} f(\pi_q \circ \omega(t))\deriv t - \int f \deriv \nu} = 0.
					\end{align*}
					This implies, by letting $f$ range through a countable dense set of functions, that for $W_q$-almost every path $\omega \in \mathsf{B}^{\delta}_{pq}(\kappa, B ; \beta, L)$, $\frac{1}{T}\omega_{*}\mathrm{Leb}_{[0,T]} \rightharpoonup \nu$ as $T \longrightarrow \infty$.
				\end{proof}
				\section{Conclusion of the proof}
				\label{sect:conclusion}
				\subsection{Proof of theorem \ref{lethm}}
				Corollary \ref{coro:genericite} establishes a result on distribution with respect to $\nu$ (with exponential contraction, as the proof shows) for a neighborhood of a $\nu$-generic point. To conclude the proof, we extend this to \emph{every} point in $P\setminus S$, and then to every point in $M\setminus S$. For this we use a reccurence argument (see Deroin-Kleptsyn \cite[proof of corollary 2.5]{deroinkleptsyn}). 
				\begin{lem}
					Let $q \in M\setminus S$. Then for $W_q$-almost every path $\gamma \in \Gamma_q$, $\gamma$ accumulates at infinity to a point which is not in $S$. 
					\label{lem:non-acc-sing}
				\end{lem}
				\begin{proof}
					It suffices to show that for every singular point $p_0 \in S$, there is a neighbourhood $U$ of $p_0$ such that whenever $q \in U$, $W_q$-almost surely, $\omega(t)$ does not converge to $q_0$ as $t \to \infty$. Choose $U$ small enough so that the foliation is defined by a holomorphic vector field $X$ which vanishes only at $q_0$.
					We distinguish between two cases: whether the singularity of $X$ at $q_0$ belongs to the Siegel domain or to the Poincaré domain, that is, whether the convex hull of the eigenvalues of $\deriv_{q_0}X$ contains zero (see for instance \cite{CamachoKuiperPalis}).
					
					First assume that we are in the case of the Poincaré domain. In this case, it is possible to linearise the vector field, that is, we can assume, reducing $U$ if necessary, that in local holomorphic coordinates $(x,y,z)$, with $q_0$ having coordinates $(0,0,0)$, the vector field $X$ is given by:
					\[X = \alpha x\partial_x + \beta y\partial_y + \gamma z\partial_z,\]
					where $\alpha, \beta, \gamma \in \C$ and zero does not belong to the convex hull of $\alpha, \beta, \gamma$. The local flow of $X$ at some point $q$ with coordinates $(x, y,z)$ is given by
					\[\zeta \mapsto \left(xe^{\alpha\zeta}, ye^{\beta\zeta}, ze^{\gamma\zeta}\right).\]
					Further reducing the neighbourhood $U$ if necessary, we can assume that $U = \{(x,y,z)\,:\,\abs{x},\abs{y},\abs{z}<1\}$. Then the flow can be defined on the domain \[\{\zeta \in \C\,:\,\mathrm{Re}(\alpha\zeta) < -\log\abs{x}, \mathrm{Re}(\beta\zeta) < -\log\abs{y}, \mathrm{Re}(\gamma\zeta) <  -\log\abs{z}\},\]
					and this gives a parametrisation of $L_q \cap U$.
					But almost surely, a Brownian path for the hyperbolic metric (or for the euclidean metric) reaches the boundary of this domain in finite time. By conformal invariance (see also below subsection \ref{subsect:simi-poincare}), this is also true for the Poincaré-type metric. Therefore, almost every Brownian path starting in $L_q \cap U$ leaves $U$ in finite time. 
					
					We now treat the case of the Siegel domain. In this case it is in general not possible to linearise holomorphically a vector field, however there are still some weaker normal forms which are sufficient for our purpose. Indeed, by a theorem of Camacho-Kuiper-Palis \cite[lemma 7]{CamachoKuiperPalis}, it is possible to choose holomorphic coordinates $(x,y,z)$ such that, reducing $U$ if necessary, we have
					\[X = \alpha x\partial_x + \beta y\partial_y + \gamma z\partial_z + xyzR,\]
					where $R$ is a holomorphic vector field, and $\alpha, \beta, \gamma \in \C$, no two of which are colinear over $\R$, and $0$ belongs to the convex hull of $\alpha, \beta, \gamma$. In this case, for every $q$ with coordinates $(x,y,z)$, if none of $x$,$y$ or $z$ is zero, the leaf $L_q$ passing through $q$ in $U$ remains at a postive distance from $0$, as is shown in the aforementioned work of Camacho-Kuiper-Palis \cite{CamachoKuiperPalis} or can be seen directly using standard estimates on solutions of ordinary differential equations. Therefore, it suffices to treat the case where one of the coordinates is zero. But in this case, the vector field is linear in the plane defined by the vanishing of that coordinate, and we can use the same argument as above.
				\end{proof}
				
				\begin{lem}
					Let $\delta_0 \in (0,1) \cap \Q$.
					Let $p \in \mathrm{supp}\,\nu$ and $U$ a neighbourhood of $p$ in $M\setminus S$. Then for every $q \in P\setminus S$, for $W_q$-almost every path $\gamma \in \Gamma_q$, there exist infinitely many $n \in \N$ such that $\gamma(n\delta_0) \in U$.
					\label{lem:recurrence1}
				\end{lem}
				\begin{proof}
					The closure of the leaf $L_q$ passing through $q$ supports a positive directed harmonic current by theorem 1.4 of Berndtsson-Sibony \cite{BerndtssonSibony}. Because such a current is unique in $P$, we must have $ \mathrm{supp}\,\nu \subset \overline{L_q} $. Therefore $L_q \cap U \neq\varnothing$. 
					By lemma \ref{lem:non-acc-sing}, for $W_q$-almost every path $\gamma \in \Gamma_q$, there exists a compact set $K \subset P\setminus S$ such that there exist infinitely many $t \geq 0$ such that $\gamma(t) \in K$. By the continuity of the diffusion operator (see Garnett \cite{Garnett}), there exists $p_0 > 0$ such that for every $q' \in K'$, and every $\tau \in [1/2\delta_0,3/2\delta_0]$, $W_{q'}\left\{\gamma\,:\,\gamma(\tau) \in K\right\} > p_0$. From this and an argument using the strong Markov property, we deduce the result.  
				\end{proof}
				
				\begin{lem}
					Let $p \in P \setminus S$. Then $W_p$-almost every path $\gamma$ is distributed with respect to $\nu$.
					
					In fact, for every $\eta > 0$, there exists a neighbourhood $U$ of $p$ in $M\setminus S$ such that for every $q \in U$, with $W_q$-probability $> 1 -\eta$, $\omega$ is distributed with respect to $\nu$.
					\label{lem:similarite2}
				\end{lem}
				\begin{proof}
					Let $\delta_0 = 1/2$ say. Let $p_0$ be a generic point to which corollary \ref{coro:similarite} applies with $\eta = 1/10$, say, and $U$ the associated neighbourhood in $P\setminus S$. Then for every $p \in P\setminus S$, by lemma \ref{lem:recurrence1}, for $W_p$-almost every path $\gamma \in \Gamma(p)$, there exists $N = N(\gamma) \in \N$ such that $\gamma(N\delta_0) \in U$. 
					
					Now for any for any $p' \in U$, we define a stopping time $N'$ on $\Gamma_{p'}$ such that $N' = \infty$ implies that the path is distributed according to $\nu$ and belongs to $\mathscr{A}_{p'}(\kappa,\beta)$ for some $\kappa, \beta\asymp 1$. 
					For this, we consider the proof of lemma \ref{lem:similarite}. The set of paths distributed according to $\nu$ given by this lemma is a set of the form $B^{\delta_0}_{\tilde{p}\tilde{q}}(\kappa, B ; \beta, \rho, L)$ for some $p$ that we can assume belongs to $U$. The condition of belonging to this set is expressed as a countable set of condition indexed by an integer $n$ (see the notations in definition \ref{def:B}), and we can consider the smallest such $n$ such that these conditions do not hold, and let $N'(\gamma) = n$.
					If no such integer exists, we let $N'(\gamma) = \infty$ and then by lemma \ref{lem:dediscretisation}, there exists some $\beta, \kappa\asymp 1$ satisfying the assumption \ref{assumption-kappa} above such that $\gamma \in \mathscr{A}_{\tilde{p}'}(\kappa,\beta)$ and $\gamma$ is distributed with respect to $\nu$.
					
					Now we define the sequence $\tau_0(\gamma) := 0$, $\tau'_1(\gamma) :=0$, and
					\[\tau_{k}(\gamma) := \tau'_{k-1}(\gamma) + N\left(\sigma_{\tau'_{k-1}(\gamma)}(\gamma)\right)\delta_0,\]
					\[\tau'_k(\gamma) := \tau_k(\gamma) + N'\left(\sigma_{\tau_k(\gamma)}(\gamma)\right)\delta_0.\]
					Then we have $\tau_0 \leq \tau_0' \leq \tau_1 < \tau'_1 \leq \tau_2 < \tau'_2 \leq \tau_3 < \tau'_3 \leq \ldots$ and morover, if some $\tau_k$ is infinite, then there exists $n\in\N$ such that $\sigma_{n\delta_0}(\gamma)$ is distributed according to $\nu$ and belongs to $\mathscr{A}_{p'}(\kappa,\beta)$ for some $\kappa, \beta\asymp 1$. 
					Let $k \in \N$, $k \geq 1$. By $2k$ applications of the strong Markov property, we have
					\[W_{p}(\tau'_k < \infty) = W_p\left\{\tau_1, \tau'_1, \ldots, \tau_k, \tau'_k < \infty\right\} \leq \eta^{k}.\]
					Therefore, almost surely there exists $k \in \N$ such that $\tau'_k = \infty$. Because the distribution of a path is shift-invariant, the first part the result is proved. Applying lemmas \ref{lem:dediscretisation} and \ref{lem:red-delta}, we deduce moreover that for every $p \in P\setminus S$ satisfies the conditions of proposition \ref{lem:similarite}, and from this we deduce the second part of the statement.
				\end{proof}
				To conclude the proof, we need the following lemma, which can be proved by \textit{topologically} linearising the vector field near the singular points. It is a consequence of (the proofs of) lemmas 7.3 and 7.4 and proposition 7.6 of Loray-Rebelo \cite{LorayRebelo}.
				\begin{lem}
					If $M = \mathbb{P}^3(\C)$, $P$ is a projective plane in $M$, and $\mathscr{F}$ has no invariant algebraic curve, then there exists a compact set $K \subset P \setminus S$ such that for every leaf $L$ of $\mathscr{F}$, $\overline{L} \cap K \neq \varnothing$.
					\label{lem:accumulation}
				\end{lem}
				Using this lemma and arguing as in the proof of lemma \ref{lem:similarite2}, we obtain the following statement:
				\begin{prop}
					Let $p \in M\setminus S$. Then $W_p$-almost every path $\gamma$ is distributed with respect to $\nu$. Moreover, $\gamma$ exponentially contracts a small ball in the sense of proposition \ref{prop:contraction}.
					\label{prop:similarite2}
				\end{prop}
				Finally we can prove theorem \ref{lethm}.
				\begin{proof}[Proof of theorem \ref{lethm}]
					Let $m$ be an ergodic harmonic probability measure for $\mathscr{F}$ on $M$. By the random ergodic theorem, for $m$-almost every point $p$, $p$ is regular and a $W_p$-almost surely, a brownian path is distributed with respect to $m$. Fix one of these points $p$. By proposition \ref{lem:similarite2}, $W_p$-almost every brownian path is distributed according to $\nu$. Therefore $m = \nu$, and this concludes the proof by the ergodic decomposition (see Nguyên \cite{NguyenSurvey}) and the correspondence between harmonic currents and harmonic measures.
					
				\end{proof}
				
				Let us also note that by modifying the proofs in section \ref{sect:similarite} using Nguyên's theorem and the ergodic theorem, we can assume that the measures which equidistribute are tight in $M\setminus S$. Indeed along a generic path, the function $\ell$ has a finite average, and therefore by the similarity arguments, this is true for almost every path starting from any fixed choice of point in $M\setminus S$. But for such a path, it is clear that we have tightness of the measures in the statement of lemma \ref{lem:similarite}. Therefore, we can replace vague convergence in $M$ by vague convergence in $M\setminus S$, that is we get that the following:
				\begin{prop}
					For every $q \in M\setminus S$ and every bounded continuous function $\varphi$ on $M\setminus S$, for $W_q$-almost every path $\omega$, we have
					\[\frac{1}{T}\int_0^T\varphi(\omega(t))\deriv t \xrightarrow[T\to\infty]{} \int \varphi\deriv\nu.\]
				\end{prop}
				We will use this in the next subsection to prove results with the Poincaré metric instead of a smooth Poincaré-type metric.
				
				\subsection{Similarity for the Poincaré metric}
				\label{subsect:simi-poincare}
				We now refine the results obtained above for the Wiener process associated to the Laplace-Beltrami operator of the Poincaré-type metric to the Wiener process associated to the Poincaré metric. In dimension $2$, as we have seen, the Laplace-Beltrami operator changes by a multplicative factor under conformal change of metric, and this implies that Brownian motion in invariant under conformal change of metric, up to time reparametrisation.
				
				Write $g_P = h_P^2g$ for some bounded continuous function $h_P$. We denote by $W_p^P$ the Wiener measure associated to the Poincaré metric, and note that by proposition \ref{prop:typepoincare}, we have $h \asymp 1$. For $p \in M \setminus S$, and $\gamma \in \Gamma_p$, consider the function, defined for $t \geq 0$ by
				\[\tau_\gamma(t) := \int_0^t h_P(\gamma(s)))^2\mathrm{d}s.\]
				This function is non-negative, increasing and almost surely diverges at infinity and is continuous, so we can consider the inverse map $\tau_{\gamma}^{-1}$. Then the conformal invariance of Brownian motion implies that the pushforward of $W_p$ by the map $\gamma \mapsto \gamma \circ \tau_\gamma^{-1}$ is $W_p^P$. 
				Note that for $t > 0$, we have $\tau_\gamma(t) \asymp t$. For any continuous path $\gamma$ starting at $p \in M\setminus S$, and $f \in \mathrm{C}_b(M\setminus S)$, we have
				\begin{align*}
				\frac{1}{T}\int_0^Tf\left(\gamma\circ\tau_\gamma^{-1}(t)\right)\deriv t
				&= \frac{1}{T}\int_0^{\tau_\gamma^{-1}(T)} f\left(\gamma(s)\right)h_P^{2}(\gamma(s))\deriv s\\
				&= \frac{\tau_\gamma^{-1}(T)}{T}\frac{1}{\tau_\gamma^{-1}(T)}\int_0^{\tau_\gamma^{-1}(T)} f\left(\gamma(u)\right)h_P^{2}(u)\deriv u.
				\end{align*}
				We note that
				\[\frac{T}{\tau_\gamma^{-1}(T)} = \frac{1}{\tau_\gamma^{-1}(T)}\int_0^{\tau_\gamma^{-1}(T)} h_P(\gamma(s)))^2\mathrm{d}s,\]
				thus almost surely, by proposition \ref{lem:similarite2}, we have
				\[\frac{1}{T}\int_0^Tf\left(\gamma\circ\tau_\gamma^{-1}(t)\right)\deriv t \xrightarrow[T\to\infty]{} \frac{1}{\int h_P^2\deriv\nu}\int fh_P^2\deriv\nu.\]
				Consequentely, recalling that we have denoted $\nu' = T \wedge \mathrm{vol}_{g_P}$, we have proved:
				\begin{prop}
					For every $p \in M\setminus S$, for $W_p^P$-almost every path $\gamma$, $\gamma$ is distributed according to $\nu'/\nu'(M)$.
				\end{prop}
				
				\subsection{Geometric Birkhoff theorem for the Poincaré metric}
				Now consider $\{D_t^P\}$ the diffusion operator associated to the Laplace-Beltrami operator of the Poincaré metric. We have, for every $p \in M\setminus S$, $f \in \mathrm{C}_b(M\setminus S)$ and $T > 0$:
				\[\frac{1}{T}\int_0^T \left(D_t^Pf\right)(p)\deriv t = \int_{\Gamma_p}\left(\frac{1}{T}\int_0^t f(\gamma(s))\deriv s\right)\deriv W_p^P(\gamma),\]
				which therefore converges as $T\to\infty$ to $\frac{1}{\nu'(M)}\int f\deriv\nu'$.
				Now we use the results of \cite[section 7]{DNSHeat}. For any point $p \in M\setminus S$ and $R > 0$ , consider $\phi_p\colon \Disk \to L_p$ a universal covering such that $\phi_p(0) = p$ (unique up to pre-composition by a rotation around $0$), and
				\[r = \frac{e^R-1}{e^R+1},\]
				\[M_R := \int_\Disk \log^{+}\left(\frac{r}{\zeta}\right)\deriv\mathrm{vol}_{P}(\zeta),\] where we have written $\log^{+} = \max\{0,\log\}$ and $\mathrm{vol}_{P}$ is the volume element for the Poincaré metric on $\Disk$. Now we define the measure $m_{p,R}$ on $M\setminus S$ by $m_{p,R} := \frac{1}{M_R}\left(\phi_p\right)_{*}\left(\log^{+}\left(\frac{r}{\zeta}\right)\deriv\mathrm{vol}_P(\zeta)\right)$. 
				\begin{prop}
					The $m_{p,R}$ converge vaguely to $\nu'/\nu'(M)$ as $R \to \infty$.
					\label{prop:moyenne-poincare}
				\end{prop}
				Indeed, the proof of \cite[lemma 7.5]{DNSHeat} implies that for every bounded continuous function $f$, we have
				\[\abs{\int f\deriv m_{p,R} - \frac{2\pi}{M_R}\int_0^{M_R/2\pi}\left(D^P_{t}f\right)(p)\deriv t} \xrightarrow[R\to\infty]{} 0,\]
				which proves the proposition by the discussion above, as $M_R \to \infty$ when $R \to \infty$.
			\bibliographystyle{plain}
			\bibliography{feuilletage}		
\end{document}